\documentclass[12pt,reqno]{amsart}

\usepackage{amsmath,amssymb,ifthen}
\usepackage{fullpage}
\usepackage{graphicx,psfrag,subfigure}
\usepackage{color}
\usepackage{moreverb}
\usepackage{amsthm}
\usepackage{mathrsfs}
\usepackage{hhline}
\usepackage{bbm}
\usepackage[english]{babel}
\usepackage{tikz}
\usetikzlibrary{calc,math}
\usepackage{pgfplots}
\usepackage[utf8]{inputenc}
\usepackage[T1]{fontenc}
\usepackage{amsfonts}
\usepackage{enumerate}
\usepackage{hyperref}

\let\div\relax
\DeclareMathOperator{\div}{\mathrm{div}}
\DeclareMathOperator*{\argmin}{argmin}

\DeclareMathOperator{\diam}{diam}

\DeclareMathOperator\osc{osc}

\newcommand{\const}[1]{C_{\text{\rm#1}}}

\newcommand{\set}[2]{\big\{#1\,:\,#2\big\}}

\newcommand{\dual}[3][]{#1\langle#2\,,\,#3#1\rangle}

\newcommand{\norm}[3][]{#1\|#2#1\|_{#3}}

\newcommand\refine{{\tt refine}}

\newcommand\bs[1]{\boldsymbol{#1}}

\newcommand\A{\mathbb{A}}

\newcommand\N{\mathbb{N}}
\newcommand\R{\mathbb{R}}
\newcommand\T{\mathbb{T}}

\newcommand\MM{\mathcal M}

\newcommand\TT{\mathcal T}

\numberwithin{equation}{section}
\numberwithin{figure}{section}
\newtheorem{theorem}{Theorem}[section]

\newtheorem{lemma}[theorem]{Lemma}

\newtheorem{algorithm}[theorem]{Algorithm}

\newtheorem{remark}[theorem]{Remark}

\newcommand*\patchAmsMathEnvironmentForLineno[1]{%
  \expandafter\let\csname old#1\expandafter\endcsname\csname #1\endcsname
  \expandafter\let\csname oldend#1\expandafter\endcsname\csname end#1\endcsname
  \renewenvironment{#1}%
     {\linenomath\csname old#1\endcsname}%
     {\csname oldend#1\endcsname\endlinenomath}}%
\newcommand*\patchBothAmsMathEnvironmentsForLineno[1]{%
  \patchAmsMathEnvironmentForLineno{#1}%
  \patchAmsMathEnvironmentForLineno{#1*}}%
\AtBeginDocument{%
\patchBothAmsMathEnvironmentsForLineno{equation}%
\patchBothAmsMathEnvironmentsForLineno{align}%
\patchBothAmsMathEnvironmentsForLineno{flalign}%
\patchBothAmsMathEnvironmentsForLineno{alignat}%
\patchBothAmsMathEnvironmentsForLineno{gather}%
\patchBothAmsMathEnvironmentsForLineno{multline}%
}
\usepackage[mathlines]{lineno}

\newcommand\RT{\bm{\mathcal{R\hspace{-0.1em}T}}} 
\newcommand\oma{\omega_\ell(\bs{a})} 

\usepackage{cancel}
\usepackage[toc,page]{appendix}
\usepackage{bm}
\usepackage[normalem]{ulem}
\usepackage[textsize=tiny,color=blue!50!white]{todonotes}
\usepackage[usenames]{xcolor}

\newcommand{\jump}[1]{[\mkern-3mu[ #1 ]\mkern-3mu]}
\newcommand\VV{\mathcal{V}}

\title{On $\MakeLowercase{p}$-robust convergence and optimality of\\ adaptive FEM  driven by equilibrated-flux estimators}

\author{Th\'eophile Chaumont-Frelet}
\address{Inria Univ. Lille and Laboratoire Paul Painlev\'e, 59655 Villeneuve-d’Ascq, France}
\email{theophile.chaumont@inria.fr}%
\author{Zhaonan Dong}
\address{Inria Paris, 48 rue Barrault, 75647 Paris, France}
\email{zhaonan.dong@inria.fr}%
\author{Gregor Gantner}
\address{Institute for Numerical Simulation, University of Bonn, Friedrich-Hirzebruch-Allee~7, 53115 Bonn, Germany \& University of Twente, Department of Applied Mathematics, P.O. Box 217, 7500 AE Enschede, The Netherlands}
\email{gregor.gantner@utwente.nl}
\author{Martin Vohral\'ik}
\address{Inria Paris, 48 rue Barrault, 75647 Paris, France \& Charles University, Faculty of Mathematics and Physics, Sokolovská 83, 186 75, Prague, Czech Republic}
\email{martin.vohralik@inria.fr}%

\thanks{
Gregor Gantner acknowledges funding by the Deutsche Forschungsgemeinschaft (DFG, German Research Foundation) under Germany's Excellence Strategy -- EXC-2047/1 -- 390685813 and the individual research grant 545527047.
Martin Vohral\'ik was supported by the French Agence Nationale de la Recherche (ANR) under grant ANR-24-CE92-0065-01 (project RANPDEs).
}


\begin{document}

\begin{abstract}
Building on existing $hp$-adaptive algorithms driven by equilibrated-flux estimators from [ESAIM Math. Model. Numer. Anal. 57 (2023), 329--366] and the references therein, we propose a novel $h$-adaptive algorithm for a fixed polynomial degree $p$. We consider a conforming finite element discretization of the Poisson equation in two or three space dimensions.
Supposing piecewise polynomial right-hand side of degree $p-1$, we show that the algorithm yields error contraction at each step, with a contraction factor that is independent of $p$ provided that a certain {\sl a posteriori} verifiable criterion is satisfied.
We further show that this algorithm converges at optimal algebraic rate $s$ if the D\"orfler marking parameter is chosen below some specified $p$-independent upper threshold. The constants involved here are $p$-robust, although they may depend on the rate $s$.
The theoretical results are supported by numerical experiments, in which the {\sl a posteriori} criterion is always satisfied for one or a few local mesh refinement steps by newest-vertex bisection.
\end{abstract}

\date{\today}
\maketitle


\section{Introduction}

\subsection{Available results}

A posteriori error estimation in finite element methods (FEM) for linear elliptic partial differential equations (PDEs) along with (optimal) convergence of corresponding adaptive mesh-refinement algorithms are nowadays well understood; cf.~\cite{ao00,repin08,verfuerth13} and~\cite{cfpp14,bcnv24}.
Regarding (optimal) convergence proofs, a particular focus has been on residual estimators, e.g., \cite{doerfler96,mns00,Mor_Noch_Sieb_stars_03,bdd04,stevenson07,dk08,ckns08}, with only few exceptions, e.g., \cite{ks11,cn12,cfpp14,Ber_Bof_Prag_Syng_a_post_20}, which, however, typically rely crucially on the local equivalence of the considered estimators to residual ones.
Such a local equivalence is known for estimators based on local flux equilibration~\cite{Dest_Met_expl_err_CFE_99,bs08,bps09,ev15}; cf.~\cite{ks11,cn12,Ern_Voh_adpt_IN_13,Ber_Bof_Prag_Syng_a_post_20}.
These equilibrated-flux estimators provide a \emph{guaranteed} upper bound as well as a polynomial-degree-robust ($p$-robust) lower bound for the approximation error~\cite{ps47,bps09,ev15,ev20}.
This is \emph{not} the case for residual estimators, and the constants in both the error-contraction results~\cite{doerfler96,mns00,Mor_Noch_Sieb_stars_03,dk08,ckns08} and the optimal-convergence results~\cite{stevenson07,ckns08} depend on the employed polynomial degree $p$.
In particular, the upper threshold for the D\"orfler marking parameter that ensures optimal convergence depends, at least in theory, on $p$.
Indeed, this threshold hinges on discrete reliability in combination with either efficiency~\cite{stevenson07} or some stability property~\cite{dk08,ckns08} of the residual estimator, both of which depend on $p$; see also~\cite{dk21}.
So, when extending convergence results from residual estimators to equilibrated-flux estimators via local equivalence as done in~\cite{ks11,cn12,Ber_Bof_Prag_Syng_a_post_20}, the involved constants inevitably depend on the polynomial degree.

In contrast, $p$-robust error contraction is indeed documented in a broad series of careful numerical computations for an $hp$-adaptive method driven by an equilibrated-flux estimator~\cite{cnsv17b}, which together with a coarsening routine even yields instance optimality~\cite{cnsv17a} for $d\in\{1,2\}$.
Similarly, \cite{desv18,dev20,dv23} propose $hp$-adaptive methods that guarantee uniform $p$-robust error contraction as long as an {\sl a posteriori} verifiable criterion is satisfied.
More precisely, besides the local flux equilibration on vertex patch subdomains $\oma$, which give rise to vertex-based error indicators $\eta_\ell(\bs{a})$ for all mesh vertices $\bs{a}\in\VV_\ell$, local residual liftings $r_{\ell+1}^{\bs{a}} \in H_0^1(\oma)$ on all  marked vertex patches are computed.
One then defines
\begin{subequations}\begin{align}\label{eq:Clb_intro}
    \const{lb}(\bs{a}) & := \frac{\eta_\ell(\bs{a})}{\norm{\nabla r_{\ell+1}^{\bs{a}}}{\oma}} \quad \text{ for all marked vertices } \bs{a}\in\MM_\ell,\\
    \const{lb}(\ell) & :=\max_{\bs{a}\in\MM_\ell} \const{lb}(\bs{a}).
\end{align}\end{subequations}
For a fixed user-defined lower-bound parameter $\const{lb,max}>0$, the criterion then reads \begin{align}\label{eq:criterion}
	  \const{lb}(\ell) \le \const{lb,max}.
\end{align}
For the Poisson model problem with homogeneous Dirichlet boundary conditions, \cite{dv23} show the existence of a $p$-independent constant $\const{lb,opt}>0$ (see~\eqref{eq:Clb_opt})
such that sufficiently small data oscillations as well as sufficiently strong refinement of all marked vertex patch subdomains $\oma$ from the current conforming triangulation $\TT_\ell$ to the next one $\TT_{\ell+1}$ guarantee that $\const{lb}(\bs{a}) \le \const{lb,opt}$.
Hence, choosing $\const{lb,max} \ge \const{lb,opt}$, criterion~\eqref{eq:criterion} can indeed be satisfied.
However, $\const{lb,opt}$ is not computable, and even if it was, the required depth of refinement is unknown and might, at least in theory, not be uniformly bounded, which would set optimal convergence out of reach;
in practical numerical experiments of~\cite{dv23} and Section~\ref{sec:numerics} below, though, $\const{lb}(\ell)$ is observed to be always smaller than $1.6$ already for one or two newest-vertex bisections.
Moreover, \cite{dv23} additionally show, under an interior node property (see~\eqref{eq:interior_node}), that $\const{lb}(\ell) \le \const{lb,int}(p)$ for a $p$-dependent constant $\const{lb,int}(p)>0$.
In any case, there holds the error contraction
\begin{align}\label{eq:error_contraction}
	\norm{\nabla (u - u_{\ell+1})}{\Omega}
	\le \underbrace{\Big(1- \frac{\theta^2}{(d+1)^2 \const{lb}(\ell)^2}\Big)^{1/2}}_{=:q_{\rm ctr}} \norm{\nabla(u - u_\ell)}{\Omega},
\end{align}
for the Galerkin approximations $u_\ell$ and $u_{\ell+1}$ on the meshes $\TT_\ell$ and $\TT_{\ell+1}$ and the D\"orfler marking parameter $\theta$, where the contraction factor $0<q_{\rm ctr}\le1$ is computable \emph{without} knowing $u_{\ell+1}$.

\subsection{The present contribution}

In the present manuscript, we propose an $h$-adaptive version of \cite{desv18,dev20,dv23} that proceeds as follows: we choose a desired upper bound $\const{lb,max}>0$ (taken as $10$ in the numerical experiments in Section~\ref{sec:numerics} below) and a maximal number $\beta_{\rm max}\in\N$ of newest-vertex bisections which ensures the interior node property~\eqref{eq:interior_node} (taken as $3$ in the numerical experiments in Section~\ref{sec:numerics} below).
Then, we are sure to either satisfy criterion~\eqref{eq:criterion} after at most $\beta_{\rm max}$ newest-vertex bisections, or to satisfy $\const{lb}(\ell) \le \const{lb,int}(p)$ thanks to the interior node property. Then, as in \cite{desv18,dev20,dv23}, we obtain the error contraction~\eqref{eq:error_contraction} with computable contraction factor $q_{\rm ctr}$, where the algorithm theoretically guarantees that $\const{lb}(\ell) \le \max\{\const{lb,max},\const{lb,int}(p)\}$ uniformly in $\ell$.
Practical numerical experiments in Section~\ref{sec:numerics} always yield small $\const{lb}(\ell)< 1.6 \leq \const{lb,max} = 10$ after $\beta_{\rm max} = 3$ newest-vertex bisections, and thus, $p$-robust error contraction since condition~\eqref{eq:criterion} is satisfied. Moreover, the computable factor $q_{\rm ctr}$ is numerically close to the actual ratio $\norm{\nabla (u - u_{\ell+1})}{\Omega} / \norm{\nabla (u - u_\ell)}{\Omega}$.

Next, we prove that the equilibrated-flux estimator satisfies discrete reliability, where, using the recent $p$-robust projector from~\cite{vohralik26}, the corresponding constant does not depend on the polynomial degree $p$.
Together with $p$-robust efficiency of the estimator and the fact that the data oscillations can be bounded by the estimator (see~\eqref{eq:Cosc}), we conclude the crucial property of ``optimality of D\"orfler marking'' (see~\eqref{eq:optimality_doerfler}) for all D\"orfler parameters $\theta$ below some threshold $\tilde\theta_{\rm opt}$ that is again independent of $p$ if the right-hand side $f$ is a piecewise polynomial of degree $p-1$.
In combination with error contraction, we finally deduce that the proposed algorithm converges at any possible algebraic rate $s$ with respect to the degrees of freedom, i.e.,
\begin{align}\label{eq:opt_cvg}
	\sup_{\ell\in\N_0}\big[[p^{d}(\#\TT_\ell-\#\TT_0+1)]^{s}\norm{\nabla(u - u_\ell)}{\Omega}\big]
	\le \const{opt} \norm{u}{\A_s}
\end{align}
for some constant $\const{opt}>0$, whenever the D\"orfler marking parameter $\theta$ is below a slight, still $p$-robust, variation $\theta_{\rm opt}$ of $\tilde\theta_{\rm opt}$ (see~\eqref{eq:theta_opt2}).
Here, $\norm{u}{\A_s}<\infty$ means that a convergence rate $s$ is possible for optimally chosen meshes $\TT_{\ell,\rm opt}$, and the term $\norm{u}{\A_s}$ essentially characterizes the multiplicative constant $C$ in the estimate $\norm{\nabla (u - u_{\ell,\rm opt})}{\Omega} \le C [p^{d}(\#\TT_{\ell,\rm opt}-\#\TT_0+1)]^{-s}$ (see~\eqref{eq:approximation}).
Note that the term $p^{d}(\#\TT_\ell-\#\TT_0+1)$ is indeed equivalent to the number of degrees of freedom.
The constant $\const{opt}$ can be written in terms of other constants which are all independent of $p$, except for the error contraction factor $q_{\rm ctr}$ from~\eqref{eq:error_contraction}, which is only conditionally $p$-robust.
However, we stress that $\const{opt}$ depends on the rate $s$, which itself might depend on $p$, as in typical situations one indeed expects the best possible rate to be $s = p/d$; on uniformly refined meshes, this then in particular recovers the standard $hp$ approximation estimate $\norm{\nabla(u - u_\ell)}{\Omega} \leq \big(\frac{h}{p}\big)^p \const{opt} \norm{u}{\A_s}$.
Together with local equivalence of the equilibrated-flux estimator to the weighted residual estimator, which we recall in Appendix~\ref{sec:equivalence}, the developed arguments further allow to demonstrate that also the standard element-based algorithm steered by the equilibrated-flux estimator stated in Appendix~\ref{sec:algorithm_element} converges R-linearly at optimal rate (though with $p$-dependent constants) provided that the corresponding D\"orfler marking parameter $\theta$ is chosen below the $p$-robust threshold $\tilde\theta_{\rm opt}$.

\subsection{Setting}

On a polygonal or polyhedral Lipschitz domain $\Omega\subset \R^d$, $d\in\{2,3\}$, we consider for given $f\in L^2(\Omega)$ the Poisson model problem with homogeneous Dirichlet boundary conditions
\begin{align}\label{eq:strong}
\begin{split}
	-\Delta u &= f \quad\text{in }\Omega,
	\\
	u &= 0 \quad\text{on }\partial\Omega.
\end{split}
\end{align}
In weak form, this problem reads as follows:
Find $u\in H_0^1(\Omega)$ such that
\begin{align}\label{eq:weak}
	\dual{\nabla u}{\nabla v}_\Omega = \dual{f}{v}_\Omega \quad\text{for all }v\in H_0^1(\Omega).
\end{align}
Here, $\dual{\cdot}{\cdot}_\omega$ stands for the $L^2(\omega)$ or $[L^2(\omega)]^d$ scalar product on a (sub)domain $\omega\subseteq\Omega$. Let $\TT_\ell$ be a conforming triangulation of $\Omega$ into compact triangles or tetrahedra with corresponding vertex set $\VV_\ell$. Let $p\ge1$ be a fixed polynomial degree.
Moreover, let $u_\ell\in \mathbb{V}_\ell:=\mathcal{P}^p(\TT_\ell)\cap H_0^1(\Omega)$ be the Galerkin approximation of $u$ in the space of continuous piecewise polynomials of degree $p$ vanishing on $\partial\Omega$, which is characterized by
\begin{align}\label{eq:discrete}
	\dual{\nabla u_\ell}{\nabla v_\ell}_\Omega = \dual{f}{v_\ell}_\Omega \quad\text{for all }v_\ell\in \mathbb{V}_\ell.
\end{align}

\subsection{Outline}
Section~\ref{sec:estimator} recalls the local construction of an equilibrated flux along with the resulting {\sl a posteriori} error estimators, which are either element- or vertex-based.
We further recall and prove several properties of the estimators, namely reliability, efficiency, discrete efficiency, and discrete reliability.
These are employed in Section~\ref{sec:optimality} to prove error contraction as per~\eqref{eq:error_contraction} and optimal convergence as per~\eqref{eq:opt_cvg} of the vertex-based adaptive Algorithm~\ref{alg:afem}, which includes the verification of optimality of D\"orfler marking.
In Section~\ref{sec:numerics}, we present numerical examples to illustrate our theoretical results and empirically verify that $\const{lb}(\ell)$ is ``small'' (below $3$) in practice.
Finally, Appendix~\ref{sec:equivalence} provides the detailed arguments that the considered vertex-based estimator is locally equivalent to the usual element-based version,
and that both are indeed locally equivalent to the residual estimator.
From this, we deduce in Appendix~\ref{sec:algorithm_element} R-linear convergence at optimal rate (though with $p$-dependent constants) with $p$-robust threshold for the D\"orfler marking parameter of a standard element-based adaptive algorithm steered by the equilibrated-flux estimator.

\section{Equilibrated-flux estimator}\label{sec:estimator}

In this section, we first recall the local construction of an equilibrated flux along with the resulting {\sl a posteriori} error estimator following~\cite{Boss_hypercirc_a_post_98,Dest_Met_expl_err_CFE_99,bs08,bps09,Par_San_Die_flux_a_post_09,Cott_Diez_Huer_strict_lin_09,Ern_Voh_adpt_IN_13,ev15}.
We further recall and prove several properties of the estimator, namely its reliability, efficiency, discrete efficiency, and discrete reliability, which will be exploited in Section~\ref{sec:optimality} to prove contraction and optimal convergence of a vertex-based adaptive algorithm.

\subsection{Equilibrated flux}
For $\bs{a}\in\VV_\ell$, we define the corresponding patch $\TT_\ell(\bs{a}):=\set{T\in\TT_\ell}{\bs{a}\in T}$ of triangles or tetrahedra sharing the vertex $\bs{a}$ and $\oma$ as the associated open subdomain.
The corresponding hat function is $\psi_\ell^{\bs{a}} \in \mathcal{P}^1(\TT_\ell)\cap H^1(\Omega)$ taking value $1$ at the vertex $\bs{a}$ and $0$ at all other vertices $\VV_\ell\setminus\{\bs{a}\}$.
Later, we will also use the notation $\TT_\ell(T)$ for the set of all elements $T' \in \TT_\ell$ sharing a vertex with $T \in \TT_\ell$ and $\omega(\mathcal{S}_\ell)\subseteq\Omega$ for the open (sub)domain corresponding to $\mathcal{S}_\ell\subseteq\TT_\ell$.
In particular, $\omega_\ell(\bs{a})$ coincides with $\omega(\TT_\ell(\bs{a}))$.
With the exterior normal $\bs{n}_{\oma}$ of $\partial\oma$, we set
\begin{align*}
	\bs{H}_0(\div;\oma):=
	\begin{cases}
	\set{\bs{\tau}\in \bs{H}(\div;\oma)}{\bs{\tau}\cdot \bs{n}_{\oma}=0\text{ on }\partial\oma} \quad \text{ if } \bs{a} \in \VV_\ell \cap \Omega,
	\\
	\set{\bs{\tau}\in \bs{H}(\div;\oma)}{\bs{\tau}\cdot \bs{n}_{\oma}=0\text{ on }\partial\oma\cap (\psi^{\bs{a}}_\ell)^{-1}(\{0\})} \quad \text{ else}.
	\end{cases}
\end{align*}
Moreover, let $\Pi_\ell^{p}:L^2(\Omega)\to \mathcal{P}^p(\TT_\ell)$ denote the $L^2$-orthogonal projection onto the space of piecewise polynomials of degree $p$.
Let $\RT^p(T) := [\mathcal{P}^p(T)]^d + \mathcal{P}^p(T) \bs{x}$ be the space of vector-valued Raviart--Thomas polynomials of order $p$ on the mesh element $T\in\TT_\ell$. The space of piecewise Raviart--Thomas polynomials on a mesh $\mathcal{S}_\ell\subseteq\TT_\ell$ of a (sub)domain $\omega\subseteq\Omega$ is then $\RT^p(\mathcal{S}_\ell) := \set{  \bs{\tau}_\ell \in  [L^2(\omega)]^d }{\bs{\tau}_\ell|_T \in \RT^p(T)$ for all $T \in \mathcal{S}_\ell}$.
We define the equilibrated flux $\bs{\sigma}_\ell\in \RT^p(\TT_\ell)\cap \bs{H}(\div;\Omega)$ in the $\bs{H}(\div;\Omega)$-conforming subspace of $\RT^p(\TT_\ell)$ by
\begin{align}\label{eq:vertex_estimator}
	\bs{\sigma}_\ell:=\sum_{\bs{a}\in\VV_\ell} \bs{\sigma}_\ell^{\bs{a}}
	\quad \text{with}\quad
	\bs{\sigma}_\ell^{\bs{a}}:=\argmin_{\substack{\bs{\tau}_\ell\in \RT^p(\TT_\ell(\bs{a}))\cap \bs{H}_0(\div;\oma)\\ \nabla{\cdot} \bs{\tau}_\ell
		= \Pi_\ell^{p}(\psi_\ell^{\bs{a}} f)-\nabla\psi_\ell^{\bs{a}}\cdot \nabla u_\ell}} \norm{\psi_\ell^{\bs{a}}\nabla u_\ell + \bs{\tau}_\ell}{\oma},
\end{align}
where we tacitly use extension by zero outside of the patch subdomains $\oma$. Note that the partition of unity $\sum_{\bs{a}\in\VV_\ell} \psi_\ell^{\bs{a}} = 1$ immediately gives $\nabla{\cdot}\bs{\sigma}_\ell = \Pi_\ell^{p} f$.

\subsection{Reliability}\label{sec:reliability}
It is well-known that the equilibrated flux $\bs{\sigma}_\ell$ gives rise to a \emph{guaranteed} upper bound for the error.

\begin{theorem}[reliability with constant one~\cite{ps47}] It holds that
\begin{align}\label{eq:reliability_one}
	\norm{\nabla (u-u_\ell)}{\Omega}\le \Big(\sum_{T\in\TT_\ell} \underbrace{\big(\norm{\nabla u_\ell + \bs{\sigma}_\ell}{T} + \frac{h_T}{\pi}\norm{f-\nabla{\cdot}\bs{\sigma}_\ell}{T}\big)^2}_{=:\eta_\ell(T)^2}\Big)^{1/2},
\end{align}
where $h_T:=\diam(T)$.
\end{theorem}

\begin{proof}
This is a simple consequence of the weak formulation~\eqref{eq:weak}, integration by parts, the fact that $\nabla{\cdot}\bs{\sigma}_\ell = \Pi_\ell^{p} f$, and the Poincar\'e inequality $\norm{v-\Pi_\ell^{p} v}{T} \leq \frac{h_T}{\pi} \norm{\nabla v}{T}$:
\begin{align*}
	\norm{\nabla (u-u_\ell)}{\Omega}
	&= \sup_{v\in H_0^1(\Omega)\setminus\{0\}} \frac{\dual{\nabla(u - u_\ell)}{\nabla v}_\Omega}{\norm{\nabla v}{\Omega}}
	\\
	&= \sup_{v\in H_0^1(\Omega)\setminus\{0\}} \frac{\dual{f - \nabla{\cdot}\bs{\sigma}_\ell}{v}_\Omega - \dual{\nabla u_\ell + \bs{\sigma}_\ell}{\nabla v}_\Omega}{\norm{\nabla v}{\Omega}}
	\\
	&\le \Big(\sum_{T\in\TT_\ell} \eta_\ell(T)^2\Big)^{1/2}.
\end{align*}
This concludes the proof.
\end{proof}

We abbreviate
\begin{align*}
	\eta_\ell(\mathcal{S}_\ell) := \Big(\sum_{T\in\mathcal{S}_\ell} \eta_\ell(T)^2\Big)^{1/2}
	\quad \text{for all }\mathcal{S}_\ell\subseteq\TT_\ell.
\end{align*}

For a triangle or tetrahedron $T\in\TT_\ell$, we denote by $\VV_T$ the corresponding set of vertices.
Using the element-wise estimates
\begin{align}\label{eq:triangle}
	\norm{\nabla u_\ell + \bs{\sigma}_\ell}{T} = \Big\|\sum_{\bs{a}\in\VV_T}(\psi_\ell^{\bs{a}}\nabla u_\ell + \bs{\sigma}_\ell^{\bs{a}})\Big\|_T
	\le \sum_{\bs{a}\in\VV_T} \norm{\psi_\ell^{\bs{a}}\nabla u_\ell + \bs{\sigma}_\ell^{\bs{a}}}{T}
\end{align}
as well as
\begin{align}\label{eq:oscillations}
	\norm{f - \nabla{\cdot}\bs{\sigma}_\ell}{T}
	&= \Big\|\sum_{\bs{a}\in\VV_T} (\psi_\ell^{\bs{a}} f - \nabla \psi_\ell^{\bs{a}}\cdot\nabla u_\ell - \nabla{\cdot}\bs{\sigma}_\ell^{\bs{a}})\Big\|_T
	\le \sum_{\bs{a}\in\VV_T} \norm{\psi_\ell^{\bs{a}} f - \Pi_\ell^{p}(\psi_\ell^{\bs{a}} f)}{T},
\end{align}
we derive the vertex-based reliability estimate
\begin{align}\label{eq:reliability}
	\norm{\nabla (u-u_\ell)}{\Omega}
	\le (d+1)^{1/2} \Big(\sum_{\bs{a}\in\VV_\ell}\underbrace{\sum_{T\in\TT_\ell(\bs{a})} \big(\norm{\psi_\ell^{\bs{a}}\nabla u_\ell + \bs{\sigma}_\ell^{\bs{a}}}{T}
		+ \frac{h_T}{\pi}\norm{\psi_\ell^{\bs{a}} f - \Pi_\ell^{p}(\psi_\ell^{\bs{a}} f)}{T} \big)^2}_{=:\eta_\ell(\bs{a})^2}\Big)^{1/2}.
\end{align}
Similarly as above, we abbreviate
\begin{align} \label{eq_eta_W_ell}
	\eta_\ell(\mathcal{W}_\ell) := \Big(\sum_{\bs{a}\in\mathcal{W}_\ell} \eta_\ell(\bs{a})^2\Big)^{1/2}
	\quad \text{for all }\mathcal{W}_\ell\subseteq\VV_\ell.
\end{align}

Finally, we introduce notation for the respective oscillations
\begin{subequations}\label{eq_osc}\begin{align}
	\osc_\ell(T) &:= \frac{h_T}{\pi} \norm{f - \Pi_\ell^{p} f}{T} \quad \text{for all }T\in\TT_\ell, \label{eq_osc_T}
	\\
	\osc_\ell(\bs{a}) &:= \Big(\sum_{T\in\TT_\ell(\bs{a})} \frac{h_T^{2}}{\pi^{2}}\norm{\psi_\ell^{\bs{a}} f - \Pi_\ell^{p}(\psi_\ell^{\bs{a}} f)}{T}^2\Big)^{1/2} \quad \text{for all }\bs{a}\in\VV_\ell \label{eq_osc_a}
\end{align}\end{subequations}
and, as above,
\begin{align*}
	\osc_\ell(\mathcal{S}_\ell) &:= \Big(\sum_{T\in\mathcal{S}_\ell} \osc_\ell(T)^2\Big)^{1/2}
	\quad \text{for all }\mathcal{S}_\ell\subseteq\TT_\ell,
	\\
	\osc_\ell(\mathcal{W}_\ell) &:= \Big(\sum_{\bs{a}\in\mathcal{W}_\ell} \osc_\ell(\bs{a})^2\Big)^{1/2}
	\quad \text{for all }\mathcal{W}_\ell\subseteq\VV_\ell.
\end{align*}

\begin{remark}[local equivalence]
While in the following, we recall and prove properties for either of the presented element-based or vertex-based equilibrated-flux estimator,
we mention that they are indeed locally equivalent, even independently of the polynomial degree $p$ in case of $f\in\mathcal{P}^{p-1}(\TT_\ell)$;
see~\eqref{eq:triangle}--\eqref{eq:oscillations} for one direction and Appendix~\ref{sec:equivalence_flux} for the other.
\end{remark}

\subsection{Efficiency}
In this section, we recall that the equilibrated flux $\bs{\sigma}_\ell$ provides a $p$-\emph{robust} (local) lower bound for the error.
Let
\begin{align}\label{eq:minimization}
	\bs{\sigma}^{\bs{a}}
	:=\argmin_{\substack{\bs{\tau}\in \bs{H}_0(\div;\oma)\\ \nabla{\cdot} \bs{\tau} = \Pi_\ell^{p}(\psi_\ell^{\bs{a}} f)-\nabla\psi_\ell^{\bs{a}}\cdot \nabla u_\ell}}
		\norm{\psi_\ell^{\bs{a}}\nabla u_\ell + \bs{\tau}}{\oma}
\end{align}
be the continuous counterpart of $\bs{\sigma}_\ell^{\bs{a}}$ from~\eqref{eq:vertex_estimator}. The crucial ingredient of the efficiency proof is the stability estimate from~\cite{bps09,ev20,cv26}
\begin{align}\label{eq:stability}
	\norm{\psi_\ell^{\bs{a}}\nabla u_\ell + \bs{\sigma}_\ell^{\bs{a}}}{\oma}
	\le \const{st} \norm{\psi_\ell^{\bs{a}}\nabla u_\ell + \bs{\sigma}^{\bs{a}}}{\oma}.
\end{align}
Let \begin{align}
H_*^1(\oma) :=
	\begin{cases}
	\set{v\in H^1(\oma)}{\dual{v}{1}_{{\oma}} = 0} &\text{ if } \bs{a}\in \VV_\ell\cap\Omega,
	\\
	\set{v\in H^1(\oma)}{v=0\text{ on }\partial\oma\setminus (\psi^{\bs{a}}_\ell)^{-1}(\{0\})} &\text{ else}.
	\end{cases}
\end{align}
Following~\cite{bps09,ev15,ev20}, we introduce the local residual lifting $\rho^{\bs{a}}\in H_*^1(\oma)$ defined via
\begin{align}\label{eq:res_lift_neumann}
	\dual{\nabla \rho^{\bs{a}}}{\nabla v}_{\oma}
	= \underbrace{\dual{\Pi_\ell^{p}(\psi_\ell^{\bs{a}} f) - \nabla \psi_\ell^{\bs{a}} \cdot \nabla u_\ell}{v}_{\oma} - \dual{\psi_\ell^{\bs{a}} \nabla u_\ell}{\nabla v}_{\oma}}%
		_{=\dual{\Pi_\ell^{p}(\psi_\ell^{\bs{a}} f)}{v}_{\oma} - \dual{\nabla u_\ell}{\nabla (\psi_\ell^{\bs{a}} v)}_{\oma}}
	\text{ for all } v \in H_*^1(\oma).
\end{align}
By reformulation of the constrained minimization problem~\eqref{eq:minimization} as saddle-point problem (see, e.g., \cite{ev15}), and using~\eqref{eq:stability}, one obtains that
\begin{align}\label{eq:est2res}
	\norm{\psi_\ell^{\bs{a}}\nabla u_\ell + \bs{\sigma}_\ell^{\bs{a}}}{\oma}
	\le \const{st} \norm{\psi_\ell^{\bs{a}}\nabla u_\ell + \bs{\sigma}^{\bs{a}}}{\oma}
	= \const{st} \norm{\nabla\rho^{\bs{a}}}{\oma}
	\quad\text{for all }\bs{a}\in\VV_\ell.
\end{align}

We will also need the Poincar\'e--Friedrichs inequality on $H_*^1(\oma)$,
\begin{align}\label{eq:PF}
	\norm{v}{\oma} \le \const{PF}(\oma) \diam(\oma) \norm{\nabla v}{\oma},
\end{align}
yielding
\begin{align}\label{eq:PFS}
	\norm{\nabla (\psi_\ell^{\bs{a}} v)}{\oma}
	\le \underbrace{\big[1 + \const{PF}(\oma) \diam(\oma) \norm{\nabla \psi_\ell^{\bs{a}}}{L^\infty(\oma)}\big]}_{=:\const{cont,PF}(\bs{a})}\norm{\nabla v}{\oma}
\end{align}
for all $\bs{a}\in\VV_\ell$ and $v\in H_*^1(\oma)$.
Here, the constants $\const{st}$ and $\const{cont,PF}:=\max_{\bs{a}\in\VV_\ell} \const{cont,PF}(\bs{a})$ depend only on the space dimension $d$ and shape regularity of $\TT_\ell$.
For the constant $\const{cont,PF}(\bs{a})$, this is stated, e.g., in~\cite{cf00,bps09}.
In particular, both involved constants $\const{st}$ and $\const{cont,PF}$ are independent of the polynomial degree $p$.
These developments lead to:

\begin{theorem}[local $p$-robust efficiency~\cite{bps09,ev15,ev20}]
For all $\bs{a}\in\VV_\ell$, it holds that
\begin{align}\label{eq:local_efficiency}
\begin{split}
	\norm{\psi_\ell^{\bs{a}}\nabla u_\ell + \bs{\sigma}_{\ell}^{\bs{a}}}{\oma}
	\le \const{st}\const{cont,PF} \norm{\nabla (u - u_\ell)}{\oma}
		+ \const{st} \osc_\ell(\bs{a}).
\end{split}
\end{align}
\end{theorem}

Later, we will exploit the following resulting bound for the full element-wise error indicator
\begin{align}\label{eq:local_efficiency2}
	\eta_\ell(T) \le \const{st}\const{cont,PF} \sum_{\bs{a}\in\VV_T} \norm{\nabla (u - u_\ell)}{\oma}
		+ (\const{st} + 1) \sum_{\bs{a}\in\VV_T} \osc_\ell(\bs{a}),
\end{align}
which follows from the element bounds~\eqref{eq:triangle}--\eqref{eq:oscillations} and \eqref{eq:local_efficiency}.
This yields the global bound
\begin{align}\label{eq:efficiency}
	\eta_\ell(\TT_\ell) \le \const{eff} \big(\norm{\nabla (u - u_\ell)}{\Omega}^{2} + \osc_\ell(\VV_\ell)^{2}\big)^{1/2}
\end{align}
with a constant $\const{eff}\geq 1$ that depends only on the space dimension $d$ and the shape regularity of $\TT_\ell$.
If $f\in\mathcal{P}^{p-1}(\TT_\ell)$ (and thus $\osc_\ell(\TT_\ell) = 0 = \osc_\ell(\VV_\ell)$), the choice $\const{eff}={(d+1)\const{st} \const{cont,PF}}$ is possible.

\subsection{Discrete efficiency}\label{sec:discrete_efficiency}

In this section, we recall that the equilibrated flux $\bs{\sigma}_\ell$ also provides a lower bound for the distance between two consecutive Galerkin approximations, which is called discrete efficiency.
We first introduce for any vertex $\bs{a}\in\VV_\ell$ the local residual lifting
$r_{\ell+1}^{\bs{a}}\in \mathbb{V}_{\ell+1}^{\bs{a}}:=\mathcal{P}^p(\TT_{\ell+1})|_{\oma}\cap H_0^1(\oma) \subseteq \mathbb{V}_{\ell+1}|_{\oma}$ defined via
\begin{align}\label{eq:res_lift_dirichlet}
	\dual{\nabla r_{\ell+1}^{\bs{a}}}{\nabla v_{\ell+1}}_{\oma} = \dual{f}{v_{\ell+1}}_{} - \dual{\nabla u_\ell}{\nabla v_{\ell+1}}_{\oma}
	\quad\text{for all } v_{\ell+1}\in \mathbb{V}_{\ell+1}^{\bs{a}}.
\end{align}
Note that, in contrast to~\eqref{eq:res_lift_neumann}, \eqref{eq:res_lift_dirichlet} is discrete, it comes with a homogeneous Dirichlet boundary condition in place of a Neumann boundary condition, and there is no weighting by the hat function $\psi_\ell^{\bs{a}}$. One could also consider a discrete version of~\eqref{eq:res_lift_neumann}, cf. \cite[Definitions~4.8 and~4.9]{dev20}, though~\eqref{eq:res_lift_dirichlet} seems more natural. We already make the important remark that $\norm{\nabla r_{\ell+1}^{\bs{a}}}{\oma}$ is monotone in the sense that it is zero when $\TT_{\ell+1}$ restricted to $\oma$ is identical to $\TT_{\ell}$ and it increases for finer $\TT_{\ell+1}$ and thus finer $\mathbb{V}_{\ell+1}^{\bs{a}}$.
Let us henceforth use the convention $x/0 = \infty$ for $x > 0$ but $0/0=0$.

\begin{lemma}[local discrete efficiency~\cite{dv23}]
Let $\MM_\ell\subseteq\VV_\ell$ be a set of marked vertices and let $\TT_\ell(\MM_\ell):=\set{T\in\TT_\ell}{T\in\TT_\ell(\bs{a}) \text{ for } \bs{a} \in \MM_\ell}$ be the corresponding set of marked elements. For any conforming refinement $\TT_{\ell+1}$ of $\TT_\ell$ with Galerkin approximation $u_{\ell+1} \in \mathbb{V}_{\ell+1}$, it holds that
\begin{align}\label{eq:discrete_efficiency}
	\norm{\nabla(u_{\ell+1} - u_\ell)}{\Omega}
	\ge \frac{1}{(d+1)^{1/2} \const{lb}(\ell)} \eta_\ell(\MM_\ell)
\end{align}
for the (possibly infinite at this stage) constant $\const{lb}(\ell)$ defined as
\begin{subequations}\label{eq:Clb}\begin{align}
\const{lb}(\bs{a}) & := \frac{\eta_\ell(\bs{a})}{\norm{\nabla r_{\ell+1}^{\bs{a}}}{\oma}} \in [0,\infty] \label{eq:Clb_1}
\quad \text{for all } \bs{a} \in \MM_\ell,\\
\const{lb}(\ell) & := \max_{\bs{a}\in\MM_\ell} \const{lb}(\bs{a}) \in [0,\infty] \label{eq:Clb_2}.
\end{align}\end{subequations}
\end{lemma}

\begin{proof}
For convenience of the reader, we sketch the proof.
For $\bs{a}\in\MM_\ell$, extend $r_{\ell+1}^{\bs{a}}\in \mathbb{V}_{\ell+1}^{\bs{a}}$ from~\eqref{eq:res_lift_dirichlet} by zero outside of $\oma$. We then define the discrete residual lifting on the marked region by
\begin{align} \label{eq:rellp}
	r_{\ell+1}:=\sum_{\bs{a}\in\MM_\ell} r_{\ell+1}^{\bs{a}} \in \mathbb{V}_{\ell+1}|_{\omega_\ell(\MM_\ell)}\cap H_0^1(\omega_\ell(\MM_\ell))
\end{align}
with $\omega_\ell(\MM_\ell) = \omega(\TT_\ell(\MM_\ell))$ the open subdomain of $\Omega$ corresponding to the marked elements $\TT_\ell(\MM_\ell)$.
Then, it holds that
\begin{align*}
	\norm{\nabla(u_{\ell+1} - u_\ell)}{\Omega}
	\ge \norm{\nabla(u_{\ell+1} - u_\ell))}{\omega_\ell(\MM_\ell)}
	&= \sup_{v_{\ell+1}\in \mathbb{V}_{\ell+1}|_{\omega_\ell(\MM_\ell)}\setminus\{0\}} \frac{\dual{\nabla(u_{\ell+1}-u_\ell)}{\nabla v_{\ell+1}}_{\omega_\ell(\MM_\ell)}}{\norm{\nabla v_{\ell+1}}{\omega_\ell(\MM_\ell)}}
	\\
	&\ge \frac{\dual{\nabla(u_{\ell+1}-u_\ell)}{\nabla r_{\ell+1}}_{\omega_\ell(\MM_\ell)}}{\norm{\nabla r_{\ell+1}}{\omega_\ell(\MM_\ell)}}.
\end{align*}
Using the definition of $u_{\ell+1}$ and $r_{\ell+1}^{\bs{a}}$, we obtain the following identities for the nominator in the last term
\begin{align*}
	\dual{\nabla(u_{\ell+1}-u_\ell)}{\nabla r_{\ell+1}}_{\omega_\ell(\MM_\ell)}
	&= \sum_{\bs{a}\in\MM_\ell} \dual{\nabla(u_{\ell+1} - u_\ell)}{\nabla r_{\ell+1}^{\bs{a}}}_{\oma}
	\\
	&= \sum_{\bs{a}\in\MM_\ell} \big(\dual{f}{r_{\ell+1}^{\bs{a}}}_{\oma} - \dual{\nabla u_\ell}{\nabla r_{\ell+1}^{\bs{a}}}_{\oma}\big)
	\\
	&=\sum_{\bs{a}\in\MM_\ell} \norm{\nabla r_{\ell+1}^{\bs{a}}}{\oma}^2.
\end{align*}
Moreover, overlapping of the vertex patches and~\eqref{eq:rellp} give that
\begin{align*}
    \norm{\nabla r_{\ell+1}}{\omega_\ell(\MM_\ell)}^2 \leq (d+1) \sum_{\bs{a}\in\MM_\ell} \norm{\nabla r_{\ell+1}^{\bs{a}}}{\oma}^2.
\end{align*}

Thus, we are lead to the computable lower bound
\begin{align}\label{eq:lower_comp_bound}
	\norm{\nabla(u_{\ell+1} - u_\ell)}{\Omega}
	\ge \frac{\sum_{\bs{a}\in\MM_\ell} \norm{\nabla r_{\ell+1}^{\bs{a}}}{\oma}^2}{\norm{\nabla r_{\ell+1}}{\omega_\ell(\MM_\ell)}}
	\ge \frac{1}{(d+1)^{1/2}} \Big(\sum_{\bs{a}\in\MM_\ell} \norm{\nabla r_{\ell+1}^{\bs{a}}}{\oma}^2\Big)^{1/2}.
\end{align}
Finally, employing~\eqref{eq_eta_W_ell} and the definition of $\const{lb}(\ell)$~\eqref{eq:Clb}, we have that
\begin{align*}
	\eta_\ell(\MM_\ell) & = \Big(\sum_{\bs{a}\in\MM_\ell} \eta_\ell(\bs{a})^2\Big)^{1/2}
    = \Big(\sum_{\bs{a}\in\MM_\ell} \frac{\eta_\ell(\bs{a})^2}{\norm{\nabla r_{\ell+1}^{\bs{a}}}{\oma}^2}\norm{\nabla r_{\ell+1}^{\bs{a}}}{\oma}^2\Big)^{1/2} \\
    & \leq \const{lb}(\ell) \Big(\sum_{\bs{a}\in\MM_\ell} \norm{\nabla r_{\ell+1}^{\bs{a}}}{\oma}^2\Big)^{1/2}.
\end{align*}
In combination with~\eqref{eq:lower_comp_bound}, this leads to the assertion~\eqref{eq:discrete_efficiency}.
\end{proof}

We now discuss two cases where we can prove that $\const{lb}(\bs{a})$ from~\eqref{eq:Clb_1} is finite. First, supposing sufficiently strong (yet unknown) refinement of each marked patch $\TT_\ell(\bs{a})$, $\bs{a} \in \MM_\ell$, and sufficiently  small oscillations (recall~\eqref{eq_osc_a} and~\eqref{eq:reliability})
\begin{align*}
	\osc_\ell(\bs{a})
	\le \frac{1}{1+2\const{st}} \norm{\psi_\ell^{\bs{a}} \nabla u_\ell + \bs{\sigma}_\ell^{\bs{a}}}{\omega_\ell^{\bs{a}}},
\end{align*}
which is in particular satisfied if $f\in\mathcal{P}^{p-1}(\TT_\ell)$, where $\osc_\ell(\bs{a})=0$, \cite[Proposition~5.1]{dv23} states that
\begin{align}\label{eq:Clb_opt}
	\eta_\ell(\bs{a}) \le \underbrace{4 \const{st}\const{cont,PF}}_{=:\const{lb,opt}\in(0,\infty)} \norm{\nabla r_{\ell+1}^{\bs{a}}}{\oma}
	\quad\text{for all } \bs{a}\in\MM_\ell.
\end{align}
Here, $\const{st}$ and $\const{cont,PF}$ are the $p$-robust (finite) constants from~\eqref{eq:est2res} and~\eqref{eq:PFS}, respectively.

Second, if for all marked vertices $\bs{a}\in\MM_\ell$,
\begin{align}\label{eq:interior_node}
	\text{interior nodes are created in all $T\in\TT_\ell(\bs{a})$ as well as all $F\in\mathcal{F}_\ell^{\rm int}(\bs{a})$},
\end{align}
where the latter set denotes all faces in $\TT_\ell(\bs{a})$ that do not lie on the boundary $\partial\oma$,
and if $f\in\mathcal{P}^{p-1}(\TT_\ell)$,
then \cite[Proposition~5.2]{dv23} still guarantee that
\begin{align}\label{eq:Clb_int}
	\eta_\ell(\bs{a}) \le \const{lb,int}(p) \norm{\nabla r_{\ell+1}^{\bs{a}}}{\oma}
	\quad\text{for all } \bs{a}\in\MM_\ell,
\end{align}
albeit with a constant $\const{lb,int}(p)\in(0,\infty)$ that now depends on the polynomial degree $p$ in addition to the space dimension $d$ and shape regularity of $\TT_\ell$ and $\TT_{\ell+1}$.
As mentioned there, the requirement~\eqref{eq:interior_node} can be relaxed if $f\in\mathcal{P}^{q}(\TT_\ell)$ for some $q<p-1$ by imposing the creation of element and face bubble functions of degree $p-q$; see also~\cite[Lemma~11]{egp20}.

In general, the constant $\const{lb}(\ell)$ from~\eqref{eq:Clb_2} can take the value $\infty$, namely when $\eta_\ell(\bs{a})>0$ and $r_{\ell+1}^{\bs{a}} =0$ (equivalent to $\nabla r_{\ell+1}^{\bs{a}} =0$), but we already stress that $\const{lb}(\bs{a}) = \eta_\ell(\bs{a})/\norm{\nabla r_{\ell+1}^{\bs{a}}}{\oma}$ is fully computable without computing $u_{\ell+1}$ so that $\const{lb}(\ell)<\infty$ can be guaranteed --- at least for $f\in\mathcal{P}^{p-1}(\TT_\ell)$ --- by the adaptive algorithm through a suffient number of refinements to achieve the interior node property~\eqref{eq:interior_node}.
In this sense, discrete efficiency~\eqref{eq:discrete_efficiency} can be viewed as always attainable and \emph{conditionally} $p$-robust; we give implementation details in Algorithm~\ref{alg:afem} below.

\subsection{Discrete reliability}
Finally, we show that the equilibrated flux $\bs{\sigma}_\ell$ from~\eqref{eq:vertex_estimator} also provides an upper bound for the distance between two consecutive Galerkin approximations, which is called discrete reliability.
In the proof, we exploit a recent local projector $I_\ell:H_0^1(\Omega)\to \mathbb{V}_\ell$ from one of the authors~\cite{vohralik26} with a $p$-robust stability constant $\const{drel}\geq1$,
\begin{align}\label{eq:stable_projection}
	\norm{\nabla (1- I_\ell) v}{\Omega} \le \const{drel} \norm{\nabla v}{\Omega}
	\quad \text{for all } v\in H_0^1(\Omega).
\end{align}
More precisely, $\const{drel}$ depends only on the space dimension $d$ and shape regularity of $\TT_\ell$.

\begin{theorem}[$p$-robust discrete reliability]
Let $\widehat\TT_\ell$ be an arbitrary conforming refinement of $\TT_\ell$ with corresponding trial space $\widehat{\mathbb{V}}_{\ell}$ and Galerkin approximation $\widehat u_\ell\in\widehat{\mathbb{V}}_\ell$ of $u$.
With the second-order patch $\TT_\ell^2(\mathcal{S}_\ell):=\bigcup_{T\in\mathcal{S}_\ell}\bigcup_{T'\in\TT_\ell(T)} \TT_\ell(T')$ of a subset $\mathcal{S}_\ell \subseteq\TT_\ell$, it holds that
\begin{align}\label{eq:discrete_reliability}
	\norm{\nabla(\widehat u_\ell - u_\ell)}{\Omega}
	\le \const{drel} \eta_\ell(\TT_\ell^2(\TT_\ell \setminus \widehat \TT_\ell)).
\end{align}
\end{theorem}

\begin{proof}\label{thm:discrete_reliability}
We start with the simple identity
\begin{align*}
	\norm{\nabla(\widehat u_\ell - u_\ell)}{\Omega}
	&= \sup_{\widehat v_\ell \in \widehat{\mathbb{V}}_\ell\setminus\{0\}} \frac{\dual{\nabla(\widehat u_\ell - u_\ell)}{\nabla \widehat v_\ell}_\Omega}{\norm{\nabla \widehat v_\ell}{\Omega}}.
\end{align*}
Galerkin orthogonality and the definition of $\widehat u_\ell$ show for the nominator that
\begin{align*}
	\dual{\nabla(\widehat u_\ell - u_\ell)}{\nabla \widehat v_\ell}_\Omega
	= \dual{\nabla(\widehat u_\ell - u_\ell)}{\nabla \underbrace{(1-I_\ell)\widehat v_\ell}_{=:\widehat e_\ell}}_\Omega
	= \dual{f}{\widehat e_\ell}_\Omega - \dual{\nabla u_\ell}{\nabla \widehat e_\ell}_\Omega.
\end{align*}
Since $\widehat e_\ell \in H_0^1(\Omega)$, integration by parts shows that $\dual{\nabla{\cdot}\bs{\sigma}_\ell}{\widehat e_\ell}_\Omega + \dual{\bs{\sigma}_\ell}{\nabla\widehat e_\ell}_\Omega = 0$,
which implies together with $\nabla{\cdot}\bs{\sigma}_\ell=\Pi_\ell^{p} f$ that
\begin{align*}
	\dual{\nabla(\widehat u_\ell - u_\ell)}{\nabla \widehat v_\ell}_\Omega
	&= \dual{f - \nabla{\cdot} \bs{\sigma}_\ell}{\widehat e_\ell}_\Omega - \dual{\bs{\sigma}_\ell + \nabla u_\ell}{\nabla\widehat e_\ell}_\Omega
	\\
	&= \dual{f - \nabla{\cdot} \bs{\sigma}_\ell}{(1-\Pi_\ell^{p})\widehat e_\ell}_\Omega - \dual{\bs{\sigma}_\ell + \nabla u_\ell}{\nabla\widehat e_\ell}_\Omega.
\end{align*}
The Cauchy--Schwarz inequality and the Poincar\'e inequality yield that
\begin{align*}
	\dual{\nabla(\widehat u_\ell - u_\ell)}{\nabla \widehat v_\ell}_\Omega
	\le \sum_{T\in\TT_\ell} \Big( \frac{h_T}{\pi} \norm{f - \nabla{\cdot} \bs{\sigma}_\ell}{T} + \norm{\bs{\sigma}_\ell + \nabla u_\ell}{T} \Big)\norm{\nabla\widehat e_\ell}{T}
	= \sum_{T\in\TT_\ell} \eta_\ell(T) \norm{\nabla\widehat e_\ell}{T}.
\end{align*}
The locality of the projector from~\cite{vohralik26} guarantees that $\widehat e_\ell =  (1-I_\ell) \widehat v_\ell = 0$ on all elements $T\in\TT_\ell$ such that $\TT_\ell^2(T):=\bigcup_{T'\in\TT_\ell(T)} \TT_\ell(T) \subseteq \TT_\ell \cap\widehat \TT_\ell$.
Hence, we see with the Cauchy--Schwarz inequality and stability~\eqref{eq:stable_projection} of the projector that
\begin{align*}
	\dual{\nabla(\widehat u_\ell - u_\ell)}{\nabla \widehat v_\ell}_\Omega
	\le \eta_\ell(\TT_\ell^2(\TT_\ell\setminus\widehat\TT_\ell)) \norm{\nabla\widehat e_\ell}{\omega(\TT_\ell^2(\TT_\ell\setminus\widehat\TT_\ell))}
	&\le \eta_\ell(\TT_\ell^2(\TT_\ell\setminus\widehat\TT_\ell)) \norm{\nabla\widehat e_\ell}{\Omega}
	\\
	&\le \const{drel} \eta_\ell(\TT_\ell^2(\TT_\ell\setminus\widehat\TT_\ell)) \norm{\nabla\widehat v_\ell}{\Omega}.
\end{align*}
This concludes the proof.
\end{proof}

\begin{remark}[$p$-robust discrete reliability of weighted-residual estimator]
Note that the arguments from the proof of Theorem~\ref{thm:discrete_reliability} also apply to the standard weighted-residual estimator, defined in~\eqref{eq:weighted-residual}.
However, we recall that the (discrete) efficiency constant for the latter estimator depends on the polynomial degree $p$ owing to the use of inverse inequalities in the proof.
\end{remark}

\section{Vertex-based adaptive algorithm, contraction \& optimal convergence}\label{sec:optimality}

In this section, we consider an $h$-adaptive version of the $hp$-adaptive algorithms investigated in~\cite{desv18,dev20,dv23},
and prove conditionally $p$-robust contraction as well as optimal convergence of this algorithm.
We use the newest-vertex bisection mesh refinement algorithm~\cite{maubach95,traxler97,stevenson08};
see also~\cite{dgs25} on how to choose suitable references edges on three-dimensional initial meshes and \cite{kpp13} for the insight that reference edges can be chosen arbitrarily on two-dimensional initial meshes.

\subsection{Vertex-based adaptive algorithm}\label{sec:algorithm_vertex}

We consider the following vertex-based adaptive algorithm.

\begin{algorithm}[vertex-based, conditionally $p$-robust]\label{alg:afem}
\textbf{Input:} Triangulation $\TT_0$, D\"orfler marking parameter $0 < \theta \le 1$, desired upper bound $\const{lb,max}\in(0,\infty)$ for the constants $\const{lb}(\bs{a})$ from~\eqref{eq:Clb_1}, and maximal number $\beta_{\rm max}$ of newest-vertex bisections applied in each refinement step, where $\beta_{\rm max}$ guarantees at least the interior node property~\eqref{eq:interior_node}, i.e., $\beta_{\rm max} \geq 3$ for $d=2$ and $\beta_{\rm max}\geq 6$ for $d=3$.\\
\textbf{Loop:} For each $\ell=0,1,2,\dots$, iterate the following steps {\rm (i)--(v)}:
\begin{enumerate}[\rm(i)]
\item Compute the Galerkin approximation $u_\ell\in\mathbb{V}_\ell$.
\item Compute the error indicators $\eta_\ell(\bs{a})$ from~\eqref{eq:reliability} for all vertices $\bs{a}\in\VV_\ell$.
\item Determine a minimal set of marked vertices $\MM_\ell\subseteq\VV_\ell$
satisfying the D\"orfler marking
\begin{align}\label{eq:doerfler}
	\theta\,\eta_\ell(\VV_\ell) \le \eta_\ell(\MM_\ell).
\end{align}
\item For each marked vertex $\bs{a}\in\MM_\ell$, initialize $\TT_{\ell+1}^{\bs{a},0}:=\TT_\ell(\bs{a})$ and iterate the following steps for each $\beta = 1,\dots, \beta_{\rm max}$:
\begin{itemize}
	\item Employ one newest-vertex bisection step for all $T\in\TT_{\ell+1}^{\bs{a},\beta-1}$ and a minimal number of additional newest-vertex bisections to ensure conformity of the resulting triangulation $\TT_{\ell+1}^{\bs{a},\beta}$.
	\item Compute $r_{\ell+1}^{\bs{a},\beta} \in \mathbb{V}_{\ell+1}^{\bs{a},\beta}:=\mathcal{P}^p(\TT_{\ell+1}^{\bs{a},\beta})\cap H_0^1(\oma)$ defined via
\begin{align*}
	\dual{\nabla r_{\ell+1}^{\bs{a},\beta}}{\nabla v_{\ell+1}}_{\oma} = \dual{f}{v_{\ell+1}}_{\oma} - \dual{\nabla u_\ell}{\nabla v_{\ell+1}}_{\oma}
	\quad\text{for all } v_{\ell+1}\in \mathbb{V}_{\ell+1}^{\bs{a},\beta}.
\end{align*}
	\item Stop the loop and set $\beta(\bs{a}):=\beta$ if
	\begin{align}\label{eq:stop_refine}
		\beta = \beta_{\rm max} \quad \text{or} \quad \const{lb}^{\beta}(\bs{a}):=\frac{\eta_\ell(\bs{a})}{\norm{\nabla r_{\ell+1}^{\bs{a},\beta}}{\oma}} \le \const{lb,max}.
	\end{align}
\end{itemize}
\item Employ a minimal number of newest-vertex bisections to obtain from $\TT_\ell$ a conforming triangulation $\TT_{\ell+1}:=\refine(\TT_\ell,\MM_\ell)$ such that $\TT_{\ell+1}|_{\oma}:=\set{T\in\TT_{\ell+1}}{T\subseteq\oma}$ is equal to or finer than $\TT_{\ell+1}^{\bs{a},\beta(\bs{a})}$ for all $\bs{a}\in\MM_\ell$.
\end{enumerate}
\textbf{Output:} Nested sequence of triangulations $\TT_\ell$, corresponding Galerkin approximations $u_\ell \in \mathbb{V}_\ell$, and
equilibrated-flux estimators $\eta_\ell(\VV_\ell)$ for all $\ell \in \N_0$.
\end{algorithm}

The marking step~(iii) can easily be realized by first sorting all error indicators $\eta_\ell(\bs{a})$, which, however, entails log-linear complexity.
Minimality of $\MM_\ell$ up to a factor $2$, which is still sufficient to prove optimal convergence, can be realized by a binning strategy in linear complexity~\cite[Section~5] {stevenson07}.
Recently, \cite{pp20} proposed a quick-selection algorithm that realizes (iii) in linear complexity.

The refinement step~(iv) is designed to guarantee several properties.
First of all, elements containing marked vertices are refined, i.e.,
$\TT_\ell(\MM_\ell) =\set{T\in\TT_\ell}{T\in\TT_\ell(\bs{a}) \text{ for } \bs{a} \in \MM_\ell} \subseteq \TT_\ell\setminus\TT_{\ell+1}$.
Next, the employment of the parameters $\const{lb,max}$ and $\beta_{\rm max}$ ensures through the condition~\eqref{eq:stop_refine} that the number of newest-vertex bisections is limited by $\beta_{\rm max}$ and that either $\const{lb}^\beta(\bs{a}) \le \const{lb,max}$ (from the second requirement from~\eqref{eq:stop_refine}), or, at least for $f\in\mathcal{P}^{p-1}(\TT_0)$, $\const{lb}^\beta(\bs{a}) \le \const{lb,int}(p)$ (by~\eqref{eq:Clb_int}, since $\beta_{\rm max}$ ensures the interior node property~\eqref{eq:interior_node}).
Recall the definition~\eqref{eq:res_lift_dirichlet} of $r_{\ell+1}^{\bs{a}}\in \mathbb{V}_{\ell+1}^{\bs{a}}:=\mathcal{P}^p(\TT_{\ell+1})|_{\oma}\cap H_0^1(\oma) \subseteq \mathbb{V}_{\ell+1}|_{\oma}$.
Since $\mathbb{V}_{\ell+1}^{\bs{a},\beta(\bs{a})}\subseteq \mathbb{V}_{\ell+1}^{\bs{a}}$ (some additional refinements may have been applied), by monotonicity, $\norm{\nabla r_{\ell+1}^{\bs{a},\beta(\bs{a})}}{\oma}\le \norm{\nabla r_{\ell+1}^{\bs{a}}}{\oma}$. Consequently, the constants $\const{lb}(\ell)$ from~\eqref{eq:Clb} become \emph{uniformly bounded} in $\ell$,
\begin{align*}
    \const{lb}(\ell) \le \max\{\const{lb,max},\const{lb,int}(p)\} \in [0,\infty) \quad \text{for all } \ell\in\N_0.
\end{align*}
Then, Algorithm~\ref{alg:afem} ensures discrete efficiency~\eqref{eq:discrete_efficiency} with constant $\max\{\const{lb,max},\const{lb,int}(p)\}\in[0,\infty)$, which is uniform in $\ell$, but potentially $p$-dependent.

\begin{remark}[choice of $\const{lb,max}$]
By~\eqref{eq:Clb_opt}, the optimal $p$-robust choice of $\const{lb,max}$ would be $4\const{st}\const{cont,PF}$, which guarantees that the second condition in~\eqref{eq:stop_refine} is reachable when $\beta_{\rm max}$ is sufficiently large, at least if $f\in\mathcal{P}^{p-1}(\TT_0)$. This would give $\const{lb}(\ell) \le \const{lb,max}$.
But $4\const{st}\const{cont,PF}$ cannot be computed exactly and might request $\beta_{\rm max}$ to grow with $p$.
If, on the other hand $\const{lb,max}$ is chosen too small, the second condition in~\eqref{eq:stop_refine} may not be reached at all.
Then, $\beta = \beta_{\rm max}$ applies and we have $\const{lb}(\ell) \le \const{lb,int}(p)$, which again might grow with $p$.
In numerical experiments in Section~\ref{sec:numerics} below, $\beta_{\rm max}=3$ and $\const{lb,max}= 10$ are used.
For these experiments, the second condition in~\eqref{eq:stop_refine}, i.e., $\const{lb}(\ell) \le \const{lb,max}= 10$, is numerically satisfied for all considered $p$; we actually observe $\const{lb}(\ell)< 1.6$ in Figure~\ref{Ex1:C_{lb}=3,4}) and Figure~\ref{Ex2:C_{lb}=3,4}.
\end{remark}

\begin{remark}[refinement strategy] \label{rem:refinement}
The only crucial property of the refinement strategy employed in Algorithm~\ref{alg:afem} is that it guarantees $\eta_\ell(\bs{a})/\norm{\nabla r_{\ell+1}^{\bs{a}}}{\oma} \lesssim 1$ for all $\bs{a}\in\MM_\ell$
at the expense of a uniformly bounded number of newest-vertex bisections (up to conforming closure) of the elements in the marked vertex patches; see \cite{egp20} for possible other strategies.
\end{remark}

\subsection{Error contraction}

It is well-known that discrete efficiency~\eqref{eq:discrete_efficiency} in combination with D\"orfler marking~\eqref{eq:doerfler} yields contraction of the error~\cite{doerfler96,mns00,cn12}.
In particular, this has been exploited in~\cite{desv18,dev20,dv23} for $hp$-adaptive versions of Algorithm~\ref{alg:afem} and leads here to a conditionally $p$-robust contraction.

\begin{theorem}[conditionally $p$-robust contraction] \label{thm:contraction}
Suppose that $f\in\mathcal{P}^{p-1}(\TT_0)$.
Then, Algorithm~\ref{alg:afem} yields that
\begin{align}\label{eq:contraction}
	\norm{\nabla (u - u_{\ell+1})}{\Omega}
	\le \underbrace{\Big(1- \frac{\theta^2}{(d+1)^2 \const{lb}^2}\Big)^{1/2}}_{=:q_{\rm ctr}} \norm{\nabla(u - u_\ell)}{\Omega}
	\quad \text{for all }\ell\in\N_0,
\end{align}
where the constant $\const{lb}$ satisfies $\const{lb} \le \max\{\const{lb,max},\const{lb,int}(p)\}$.
If the second condition from~\eqref{eq:stop_refine} is satisfied for all marked vertices $\bs{a}\in\MM_\ell$ and for all $\ell\in\N_0$, then the $p$-robust bound $\const{lb} \le \const{lb,max}$ holds true.
\end{theorem}

\begin{proof}
Note that discrete efficiency~\eqref{eq:discrete_efficiency} is satisfied with $\const{lb}=\max\{\const{lb,max},\const{lb,int}(p)\}$ owing to~\eqref{eq:stop_refine} and~\eqref{eq:Clb_int}.
Galerkin orthogonality shows the Pythagoras identity
\begin{align}\label{eq:pythagoras}
	\norm{\nabla (u - u_{\ell+1})}{\Omega}^2 = \norm{\nabla(u - u_\ell)}{\Omega}^2 - \norm{\nabla(u_{\ell+1} - u_\ell)}{\Omega}^2.
\end{align}
Discrete efficiency~\eqref{eq:discrete_efficiency}, D\"orfler marking~\eqref{eq:doerfler}, and reliability~\eqref{eq:reliability} show that
\begin{align*}
	\norm{\nabla(u_{\ell+1} - u_\ell)}{\Omega}^2
	\ge \frac{1}{(d+1) \const{lb}^2} \eta_\ell(\MM_\ell)^2
	&\ge \frac{\theta^2}{(d+1) \const{lb}^2} \eta_\ell(\VV_\ell)^2
	\ge \frac{\theta^2}{(d+1)^2 \const{lb}^2} \norm{\nabla(u - u_\ell)}{\Omega}^2.
\end{align*}
This concludes the proof.
\end{proof}

\subsection{Optimality of Dörfler marking \& comparison lemma}

So far, we have seen that Dörfler marking \eqref{eq:doerfler} in the adaptive algorithm (and $f\in\mathcal{P}^{p-1}(\TT_\ell)$) implies contraction~\eqref{eq:contraction} of the error.
The next lemma  essentially states the converse implication for element-based D\"orfler marking.
In other words, Dörfler marking is not only sufficient for contraction, but in some sense even necessary.
This insight goes back to the seminal work~\cite{stevenson07}; see also~\cite{ckns08,ks11,cn12}.
Since the vertex-based oscillations require special treatment, we include a detailed proof.
In particular, we use the bound
\begin{align}\label{eq:Cosc}
	\osc_\ell(\bs{a}) \le \widetilde C_{\rm osc} \eta_\ell(\TT_\ell(\bs{a}))
	\quad \text{for all } \bs{a}\in\VV_\ell,
\end{align}
with a constant $\widetilde C_{\rm osc}$ that depends on the space dimension $d$, the shape regularity of $\TT_\ell$, and the polynomial degree $p$; see Remark~\ref{rem:Cosc} below.
We denote by $\T$ the set of all possible conforming newest-vertex bisection refinements of the initial mesh $\TT_0$.
Note that the meshes in $\T$ are uniformly shape-regular.

\begin{lemma}[$p$-robust optimality of D\"orfler marking]\label{lem:optimality_doerfler}
Abbreviate $\const{osc}:=(d+1)^{1/2} \widetilde C_{\rm osc}$.
Under the assumption
\begin{align}\label{eq_tt}
	0 < \theta < \tilde\theta_{\rm opt} := \frac{1}{\const{eff} (\const{drel}^{2} + \const{osc}^{2})^{1/2}},
\end{align}
any
\begin{align}\label{eq_qt}
	0<q_\theta^2\le \frac{1-\const{eff}^2  (\const{drel}^2 + \const{osc}^2 )\theta^2}{\const{eff}^2 (1 + \const{osc}^2)}
\end{align}
satisfies the implication
\begin{align}\label{eq:optimality_doerfler}
	\Big(\widehat\eta_\ell(\widehat\TT_\ell)
	\le q_\theta \eta_\ell(\TT_\ell)
	\,\,\,\Longrightarrow \,\,\,
	\theta \eta_\ell(\TT_\ell) \le \eta_\ell(\TT_\ell^2(\TT_\ell \setminus \widehat\TT_\ell))\Big)
	\text{ for all $\TT_\ell\in\T$ and refinements $\widehat\TT_\ell\in\T$}.
\end{align}
Note that only $\const{osc}$ potentially depends on the polynomial degree $p$, which, however, can be chosen as $\const{osc} = 0$ if $f\in\mathcal{P}^{p-1}(\TT_\ell)$, since then $\osc_\ell(\VV_\ell) = 0$.
\end{lemma}

\begin{proof}
Let $0 < q_\theta < 1$ be for the moment a free parameter. Assume that
$\widehat\eta_\ell(\widehat\TT_\ell) \le q_\theta \eta_\ell(\TT_\ell)$.
According to efficiency~\eqref{eq:efficiency}, reliability~\eqref{eq:reliability_one} with constant $1$, and the oscillation bound
$\widehat\osc_\ell(\widehat\VV_\ell) \le \const{osc} \widehat\eta_\ell(\widehat\TT_\ell)$, being a consequence of~\eqref{eq:Cosc}, this proves that
\begin{align*}
	\big[\const{eff}^{-2} - (1 + \const{osc}^2) q_\theta^2\big]  \eta_\ell(\TT_\ell)^2
	&\le \big[ \norm{\nabla(u - u_\ell)}{\Omega}^2 + \osc_\ell(\VV_\ell)^2 \big] - (1 + \const{osc}^2) \widehat\eta_\ell(\widehat\TT_\ell)^2
	\\
	& \le \big[ \norm{\nabla(u - u_\ell)}{\Omega}^2 + \osc_\ell(\VV_\ell)^2 \big] - \big[ \norm{\nabla(u - \widehat u_\ell)}{\Omega}^2 + \widehat\osc_\ell(\widehat\VV_\ell)^2 \big].
\end{align*}
Galerkin orthogonality and the Pythagoras theorem show that
\begin{align*}
	\norm{\nabla(u - u_\ell)}{\Omega}^2 - \norm{\nabla(u - \widehat u_\ell)}{\Omega}^2
	= \norm{\nabla(\widehat u_\ell - u_\ell)}{\Omega}^2.
\end{align*}
Since the oscillations coincide for vertices $\bs{a}\in\VV_\ell$ with non-refined patch $\TT_\ell(\bs{a})$, the oscillation bound \eqref{eq:Cosc} gives that
\begin{align*}
	\osc_\ell(\VV_\ell)^2 -  \widehat\osc_\ell(\widehat\VV_\ell)^2
	&\le \osc_\ell\big(\set{\bs{a}\in\VV_\ell}{\TT_\ell(\bs{a}) \cap (\TT_\ell \setminus \widehat\TT_\ell) \neq \emptyset}\big)^2
	\\
	&\le \const{osc}^2 \eta_\ell\big(\bigcup\set{\TT_\ell(\bs{a})}{\bs{a}\in\VV_\ell, \, \TT_\ell(\bs{a}) \cap (\TT_\ell \setminus \widehat\TT_\ell) \neq \emptyset}\big)^2
	\\
	&\le \const{osc}^2 \eta_\ell(\TT_\ell^2(\TT_\ell \setminus \widehat\TT_\ell))^2.
\end{align*}
Combining the last three (in)equalities with the discrete reliability~\eqref{eq:discrete_reliability}, we are led to
\begin{align*}
	\big[ \const{eff}^{-2} - (1+\const{osc}^2) q_\theta^2 \big]  \eta_\ell(\TT_\ell)^2
	&\le \norm{\nabla(\widehat u_\ell - u_\ell)}{\Omega}^{2} + \const{osc}^2 \eta_\ell(\TT_\ell^2(\TT_\ell \setminus \widehat\TT_\ell))^2
	\\
	&\le (\const{drel}^2 + \const{osc}^2)   \eta_\ell(\TT_\ell^2(\TT_\ell \setminus \widehat\TT_\ell))^2.
\end{align*}
Hence,
\begin{align} \label{eq_oD}
	\frac{1 - \const{eff}^{2}(1 + \const{osc}^2) q_\theta^2 }{\const{eff}^{2}(\const{drel}^2 + \const{osc}^2 )}  \eta_\ell(\TT_\ell)^2
	\le \eta_\ell(\TT_\ell^2(\TT_\ell \setminus \widehat\TT_\ell))^2.
\end{align}
For any $0<\theta < \tilde\theta_{\rm opt}$ from~\eqref{eq_tt}, we may now fix $q_\theta$ as per~\eqref{eq_qt} such that $0 < q_\theta < 1$ and
\begin{align*}
	\theta^2 &\le \frac{1 - \const{eff}^2 (1 + \const{osc}^2) q_\theta^2}{\const{eff}^2  (\const{drel}^2 + \const{osc}^2 )}
	< \frac{1}{\const{eff}^2  (\const{drel}^2 + \const{osc}^2)}
	= \tilde\theta_{\rm opt}^2,
\end{align*}
using also the facts that $\const{eff} \geq 1$ and $\const{drel} \geq 1$.
Using the first inequality above and~\eqref{eq_oD} concludes the proof.
\end{proof}

For $s>0$, define
\begin{align}\label{eq:approximation}
	\norm{u}{\A_s} := \sup_{N\in\N_0} \min_{\TT_\star\in\T_N}\big[[p^{d}(N+1)]^s\norm{\nabla(u - u_\star)}{\Omega}\big]
\end{align}
with
\begin{align}
	\T_N:=\set{\TT\in\T}{\#\TT-\#\TT_0\le N}.
\end{align}
While $\norm{u}{\A_s}$ is usually defined without the factor $p^{ds}$ in the literature, e.g., \cite{stevenson07,ckns08,cfpp14}, we include it
so that the term $p^{d}(N+1)$ is equivalent to the number of degrees of freedom rather than the number of mesh elements, with constants that are independent of the polynomial degree $p$.
Note that $\norm{u}{\A_s}<\infty$ means that a convergence rate $s$ of the error with respect to the number of degrees of freedom is possible for optimally chosen meshes.

The next lemma, known as the comparison lemma in the literature, provides the key ingredient for the proof of optimal convergence.
It is a simple consequence of (quasi-)monotonicity of the Galerkin error and the overlay property of newest-vertex bisection;  see, e.g., \cite[Lemma~4.14]{cfpp14}.

 \begin{lemma}[comparison lemma~\cite{stevenson07}]\label{lem:comparison}
Let $\epsilon>0$ and $s>0$ with $\norm{u}{\A_s}<\infty$.
Then, for all $\TT_\ell\in\T$, there exists a refinement $\widehat\TT_\ell\in\T$ of $\TT_\ell$ such that
\begin{subequations}
\label{eq:comparison}
\begin{align} \label{eq:comparison1}
	\norm{\nabla(u - \widehat u_\ell)}{\Omega} &\le  \epsilon,
	\\ \label{eq:comparison2}
	p^d(\#\widehat\TT_\ell-\#\TT_\ell)&\le
	\norm{u}{\A_s}^{1/s}\,\epsilon^{-1/s}.
\end{align}
\end{subequations}
\end{lemma}

\subsection{Optimal convergence}

The following theorem states optimal convergence of Algorithm~\ref{alg:afem} with respect to the number of degrees of freedom.
While the proof is standard, e.g., \cite[Section~8]{cfpp14} or \cite{ks11,cn12,Ber_Bof_Prag_Syng_a_post_20}, the novelty is conditional $p$-robustness of all involved constants provided that the second condition in~\eqref{eq:stop_refine} is satisfied for all $\ell\in\N_0$ and $\bs{a}\in\MM_\ell$; see also Remark~\ref{rem:s-dependence} below.

\begin{theorem}[optimal convergence]\label{thm:optimal_convergence}
Suppose that $f\in\mathcal{P}^{p-1}(\TT_0)$.
Then, for all
\begin{align}\label{eq:theta_opt2}
	0 < \theta < \theta_{\rm opt} := \frac{\tilde \theta_{\rm opt}}{(d+1) \const{st}\const{cont,PF}}
\end{align}
and all $s>0$, there exist constants $c_{\rm opt},\const{opt}>0$ such that the meshes $(\TT_\ell)_{\ell\in\N_0}$ generated by Algorithm~\ref{alg:afem} satisfy that
\begin{align}\label{eq:optimal_convergence}
	c_{\rm opt} \norm{u}{\A_s}
	\le\sup_{\ell\in\N_0}\big[[p^{d}(\#\TT_\ell-\#\TT_0+1)]^{s}\norm{\nabla(u - u_\ell)}{\Omega}\big]
	\le \const{opt} \norm{u}{\A_s}.
\end{align}
The constant $\const{opt}$ is given as
\begin{align}
	\const{opt} := \frac{2^s \const{eff}^2 (1-q_{\rm ctr}^{1/s})^{-s}\const{clos}^s (d+1)^s\const{nei}^{2s}}{[1-\const{eff}^2 \const{drel}^2 (d+1)^2 \const{st}^2\const{cont,PF}^2 \theta^2]^{1/2}}.
\end{align}
Here, $\const{clos}$ is the constant from the mesh closure estimate~\eqref{eq:closure} below and $\const{nei}$ is the maximal number of neighbors of an element.
Note that only $q_{\rm ctr}$ might depend on the polynomial degree $p$ --- if the second condition in~\eqref{eq:stop_refine} is not always satisfied.
The constant $c_{\rm opt}$ depends only on $ \#\TT_0$, $\const{child}$ from~\eqref{eq:child} below,  $s$, and, if there exists $\ell_0$ with $u_{\ell_0} = u$, also on $\ell_0$.
\end{theorem}

\begin{proof}
The first inequality in~\eqref{eq:optimal_convergence} is an elementary consequence of the child estimate~\cite{gss14}
\begin{align}\label{eq:child}
	\#\TT_{\ell+1} \le \const{child} \#\TT_\ell \quad\text{for all }\ell\in\N_0
\end{align}
with a constant $\const{child}>0$ that depends on $\TT_0$ and $\beta_{\rm max}$; see \cite[Proposition~4.15]{cfpp14} for the detailed argument.

For convenience of the reader, we provide the (standard) proof of the second inequality in~\eqref{eq:optimal_convergence}.
Without loss of generality, we assume that $\norm{u}{\A_s}<\infty$.
If $\norm{\nabla (u - u_{\ell_0})}{\Omega}=0$ for some $\ell_0\in\N_0$, then $\norm{\nabla(u - u_\ell)}{\Omega}=0$ for all $\ell\geq \ell_0$.
Moreover $[p^{d}(\#\TT_0-\#\TT_0+1)]^{s} \norm{\nabla (u - u_0)}{\Omega}\le \norm{u}{\A_s}$ is trivially satisfied.
Thus, it is sufficient to consider $0<\ell< \ell_0$ or $0<\ell$ if no such $\ell_0$ exists.

Let $\ell\in\N_0$ be fixed and let $j< \ell$.
For a given $\theta$ satisfying~\eqref{eq:theta_opt2}, abbreviate $\tilde\theta:=(d+1) \const{st}\const{cont,PF} \theta<\tilde\theta_{\rm opt}$.
We apply \eqref{eq:comparison} for the mesh $\TT_j$, where we choose $\epsilon = \const{eff}^{-1} q_{\tilde\theta}\norm{\nabla (u - u_j)}{\Omega}$ with (the maximal) $q_{\tilde \theta}$ of Lemma~\ref{lem:optimality_doerfler}.
In particular, \eqref{eq:comparison1} together with efficiency~\eqref{eq:efficiency} and reliability~\eqref{eq:reliability_one} yields that
\begin{align*}
	\widehat \eta_j(\widehat\TT_j)
	\le \const{eff} \norm{\nabla(u - \widehat u_j)}{\Omega}
	\le q_{\tilde\theta} \norm{\nabla(u - u_j)}{\Omega}
	\le q_{\tilde\theta} \eta_j(\TT_j).
\end{align*}
Hence, the implication~\eqref{eq:optimality_doerfler} shows that $\TT_j^2(\TT_j \setminus \widehat \TT_j)$ from discrete reliability~\eqref{eq:discrete_reliability} satisfies the Dörfler marking
$\tilde \theta \eta_j(\TT_j)\le\eta_j(\TT_j^2(\TT_j \setminus \widehat \TT_j))$.
Together with the local equivalences~\eqref{eq:ver2el_2} and \eqref{eq:el2ver} from below, this gives that
\begin{align*}
	\theta \eta_j(\VV_j) & \le (d+1)^{1/2} \const{st}\const{cont,PF} \theta \eta_j(\TT_j)
	= (d+1)^{-1/2} \tilde\theta \eta_j(\TT_j)
	\le (d+1)^{-1/2} \eta_j(\TT_j^2(\TT_j \setminus \widehat \TT_j))
	\\
	& \le \eta_j\big(\underbrace{\set{\bs{a}\in\VV_j}{\bs{a}\in \VV_T, \, T \in \TT_j^2(\TT_j \setminus \widehat \TT_j)}}_{=:\widetilde\MM_j}\big).
\end{align*}
Minimality of the set of marked vertices $\MM_j$ satisfying this D\"orfler marking (Algorithm~\ref{alg:afem} (iii)) yields that
\begin{align*}
	\#\MM_j \le \#\widetilde\MM_j \le (d+1) \#\TT_j^2(\TT_j\setminus\widehat\TT_j)
	\le (d+1)\const{nei}^2 (\#\widehat \TT_j - \#\TT_j).
\end{align*}
Thus,  \eqref{eq:comparison2} shows that
\begin{align*}
	\#\MM_j
	\le (d+1)\const{nei}^2 (\#\widehat \TT_j-\#\TT_j)
	\le (d+1)\const{nei}^2 \const{eff}^{1/s} p^{-d}\norm{u}{\A_s}^{1/s}q_{\tilde\theta}^{-1/s}\norm{\nabla (u - u_j)}{\Omega}^{-1/s}.
\end{align*}
The fact that $\#\TT_\ell>\#\TT_0$ and the mesh closure estimate~\cite{bdd04,stevenson08,kpp13} (depending only on $\TT_0$ and $\beta_{\rm max}$) thus yield  that
\begin{align}\label{eq:closure}\begin{split}
	\#\TT_\ell-\#\TT_0+1
	& \leq 2(\#\TT_\ell-\#\TT_0)
	\leq 2\const{clos} \sum_{j=0}^{\ell-1}\#\MM_j
	\\
	&\leq 2\const{clos} (d+1)\const{nei}^2\const{eff}^{1/s} p^{-d} \norm{u}{\A_s}^{1/s}q_{\tilde\theta}^{-1/s}\,\sum_{j=0}^{\ell-1}\norm{\nabla (u - u_j)}{\Omega}^{-1/s}.
\end{split}\end{align}
Contraction~\eqref{eq:contraction} and elementary calculus imply that
\begin{align*}
	\sum_{j=0}^{\ell-1}\norm{\nabla (u - u_j)}{\Omega}^{-1/s}\le (1-q_{\rm ctr}^{1/s})^{-1} \norm{\nabla (u - u_\ell)}{\Omega}^{-1/s};
\end{align*}
see~\cite[Lemma~4.9]{cfpp14}.
Overall, we thus end up with
\begin{align*}
	[p^{d}(\#\TT_\ell-\#\TT_0+1)]^{s} \norm{\nabla(u - u_\ell)}{\Omega}
	\leq 2^s (1-q_{\rm ctr}^{1/s})^{-s}\const{clos}^s (d+1)^s\const{nei}^{2s} \const{eff} \, q_{\tilde\theta}^{-1} \norm{u}{\A_s}.
\end{align*}
This concludes the proof.
\end{proof}

\begin{remark}[dependence on regularity $s$] \label{rem:s-dependence}
While, up to $q_{\rm ctr}$, all constants in Theorem~\ref{thm:optimal_convergence} are $p$-robust, and also $q_{\rm ctr}$ is $p$-robust provided that the second condition in~\eqref{eq:stop_refine} is always satisfied,
there is a dependence on the regularity $s$.
We stress that under certain assumptions, e.g., \cite{gm09}, it is possible to show that the choice $s=p/d$ yields $\norm{u}{s} < \infty$.
In this case, the constants in Theorem~\ref{thm:optimal_convergence} implicitly depend on $p$ through $s$.
\end{remark}

\begin{remark}[non-polynomial $f$] \label{rem:standard_proof}
If $f\not\in\mathcal{P}^{p-1}(\TT_0)$, then \eqref{eq:optimal_convergence} is still true (with a similar proof) for the estimator $\eta_\star(\TT_\star)$ instead of the error $\norm{\nabla(u - u_\star)}{\Omega}$ provided that linear convergence
\begin{align}\label{eq:linear_convergence}
	\eta_{\ell+n}(\TT_{\ell+n}) \le \const{lin} q_{\rm lin}^n \eta_\ell(\TT_\ell)
	\quad \text{for all }\ell,n\in\N_0
\end{align}
holds for some uniform $\const{lin}>0$ and $0<q_{\rm lin}<1$.
Linear convergence is always satisfied even for standard refinement, i.e., $\beta_{\rm max}=1$ in Section~\ref{sec:algorithm_vertex},
owing to local equivalence (see Appendix \ref{sec:equivalence_flux}--\ref{sec:equi2res}) of the equilibrated-flux estimators $\eta_\star(\TT_\star)$ and $\eta_\star(\VV_\star)$ to the weighted residual estimator $\zeta_\star(\TT_\star)$; see  \cite[Section~8]{cfpp14} or \cite{ks11,cn12,Ber_Bof_Prag_Syng_a_post_20} for the standard argument.
However, in contrast to Theorem~\ref{thm:contraction}, this argument inevitably leads to $p$-dependent constants $\const{lin}$ and $q_{\rm lin}$.
As a matter of fact, our proof of optimality of D\"orfler marking (Lemma~\ref{lem:optimality_doerfler}) 
also yields potentially $p$-dependent constants if $f\not\in\mathcal{P}^{p-1}(\TT_0)$ because of the use of the oscillation bound~\eqref{eq:Cosc}.
We further mention that, in case of $f\not\in\mathcal{P}^{p-1}(\TT_0)$, the estimator is at least equivalent to the total error, i.e.,
\begin{align}\label{eq:releff}
	\norm{\nabla(u - u_\star)}{\Omega} +\osc_\star(\TT_\star) \le 2 \eta_\star(\TT_\star) \lesssim \norm{\nabla(u - u_\star)}{\Omega} + \osc_\star(\TT_\star) \quad\text{for all }\TT_\star \in\T,
\end{align}
which is a consequence of equivalence to the weighted residual estimator $\zeta_\star(\TT_\star)$ and efficiency of the latter.
Standard arguments even show that
\begin{align}\label{eq:equivalent_classes}
	\norm{u}{\A_s} \le \norm{\eta}{\A_s} \lesssim \norm{u}{\A_s} + \norm{{\rm osc}}{\A_s},
\end{align}
where $\norm{\eta}{\A_s}$ and $\norm{{\rm osc}}{\A_s}$ are defined as in~\eqref{eq:approximation} for the estimator $\eta_\star(\TT_\star)$ and the oscillations $\osc_\star(\TT_\star)$;
see, e.g., \cite[Proposition~4.6]{cfpp14}.
\end{remark}

\begin{remark}[general diffusion matrix]
The results of this manuscript readily extend to diffusion problems $-\nabla\cdot (\bs{A} \nabla u) = f$ with $\TT_0$-piecewise constant uniformly positive definite diffusion matrix $\bs{A}\in\R^{d\times d}$. 
Also an extension to $\TT_0$-piecewise polynomial $\bs{A}$ of some fixed degree $q$ is possible if the polynomial degree of the flux reconstruction $\bs{\sigma}_\ell$ is chosen as $p+q$; see, e.g., the arguments in~\cite[Section~5]{bdfgps26}.
Then, multiplicative constants depend additionally on $q$, and the analysis also allows for $f\in\mathcal{P}^{p+q-1}(\TT_0)$. 
For general diffusion matrices and right-hand sides, one could include additional loops in Algorithm~\ref{alg:afem} to resolve them; see, e.g., \cite{stevenson07,bcnv24}. 
However, a detailed analysis of this case goes beyond the scope of the manuscript. 
\end{remark}

\section{Numerical experiments}\label{sec:numerics}

In this section, we present numerical examples to illustrate our theoretical results.
We consider Algorithm~\ref{alg:afem} with $\beta_{\rm max} = 3$ and $\const{lb,max}=10$.
It starts from a coarse mesh consisting of $6$ triangles for Example~1 and $24$ triangles for Example~2 and uses the error indicators $\eta_\ell(\bs{a})$ from~\eqref{eq:reliability} to mark vertex patches for refinement through the D\"orfler marking~\eqref{eq:doerfler} with $\theta=0.3$.

\subsection{Example 1: Adaptive FEM on L-shape with known solution}

We select $f=0$ and Dirichlet boundary conditions on the L-shaped domain $\Omega := (-1,1)^2 \setminus \big([0,1]\times[-1,0]\big)$, so that the exact solution is in polar coordinates
\begin{equation}
u(r,\vartheta) = r^{\frac23} \sin(2\vartheta/3), \qquad \vartheta\in (0,\frac{3\pi}{2}).
\end{equation}

We mention that the inhomogeneous Dirichlet boundary condition is imposed by computing an $hp$-optimal $H^1$-conforming projection-based interpolation of the boundary datum onto the finite element space; we refer to \cite[Section 5.1]{DemkowiczBuffa05} for details. 
We do not take the resulting oscillation terms directly into account in the estimator and the algorithm, as the corresponding computed values are negligibly small (not displayed).
Note that the local patch problems associated with Dirichlet boundary vertices does not require any modification. We refer to \cite[Definition 3.5]{DolejsiErnVohralik16} for further details. 

We first test the convergence rate of the above adaptive algorithm with polynomial degrees $p\in\{1,2,3,4\}$.
The energy error $\norm{\nabla(u - u_\ell)}{\Omega}$ and the {\sl a posteriori} error estimator $\eta_\ell(\TT_\ell)$ are reported in Figure~\ref{Ex1:error_estimator_p=1,2,3,4} as a function of DoFs.
We observe that the error and the estimator converge at the optimal rate $\mathcal{O}(\textup{DoFs}^{-\frac{p}{2}})$.

\begin{figure}[!tb]
\begin{center}
\begin{tabular}{cc}
\includegraphics[width=0.46\linewidth]{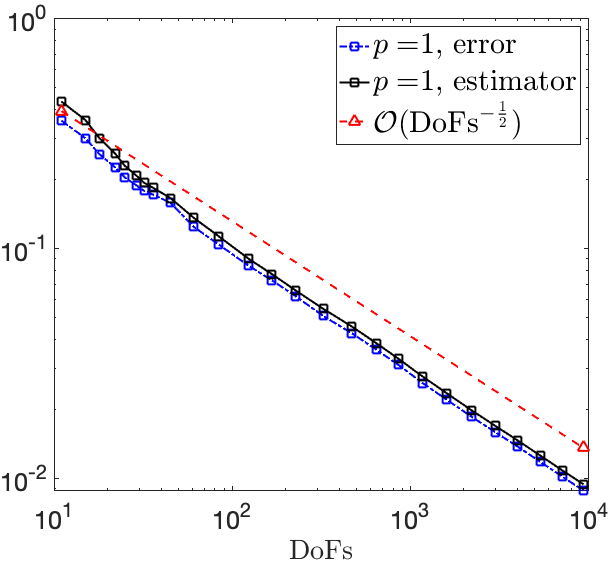} &
\includegraphics[width=0.46\linewidth]{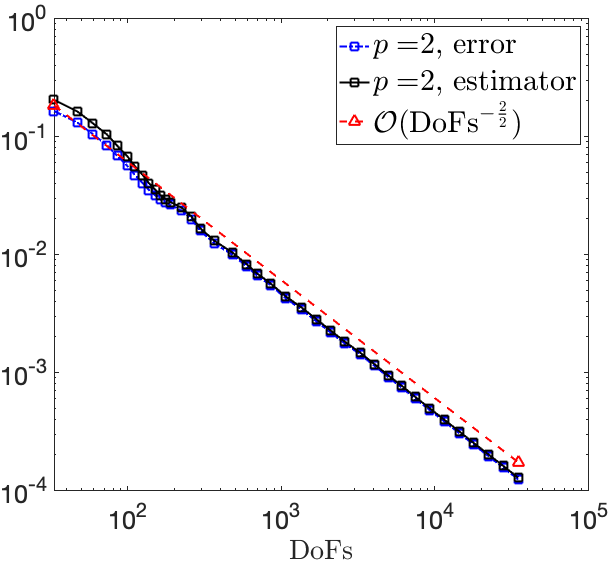} \\
\includegraphics[width=0.46\linewidth]{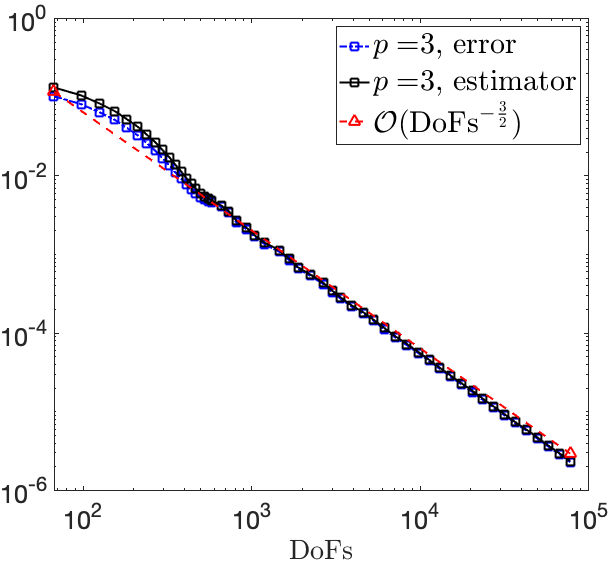} &
\includegraphics[width=0.46\linewidth]{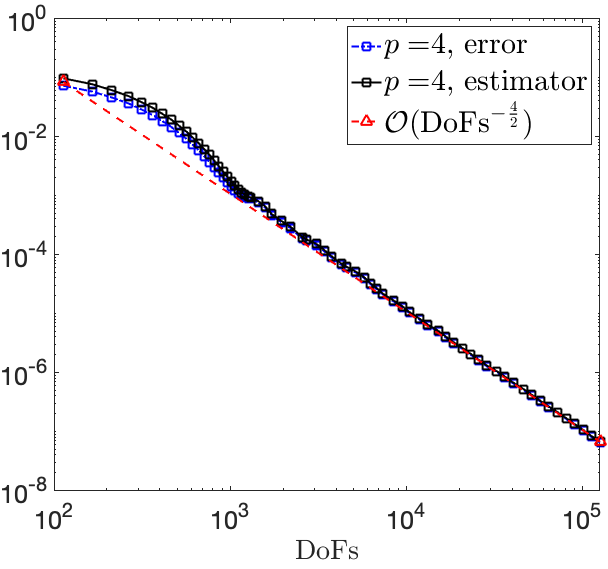}
\end{tabular}
\end{center}
\caption{Example 1. Energy error and {\sl a posteriori} error estimator as a function of DoFs for polynomial degrees $p\in\{1,2,3,4\}$.}\label{Ex1:error_estimator_p=1,2,3,4}
\end{figure}

Moreover, we observe in the left panel of Figure~\ref{Ex1:effectivity p=1,2,3,4} that the effectivity index of the estimator, i.e., the ratio $\eta_\ell(\TT_\ell) / \norm{\nabla(u - u_\ell)}{\Omega}$, is very close to the optimal value of $1$ for all polynomial degrees $p\in\{1,2,3,4\}$.
To gain further insight, we report in the right panel of Figure~\ref{Ex1:effectivity p=1,2,3,4} the effectivity index as a function of the polynomial degree $p\in\{1,\dots,13\}$ on two meshes consisting of $12$ and $48$ triangles.
The values lie between $1.2$ and $1.5$, even in this pre-asymptotic mesh regime, which numerically confirms the $p$-robustness of the equilibrated-flux estimator.

\begin{figure}[!tb]
\begin{center}
\includegraphics[width=0.45\linewidth]{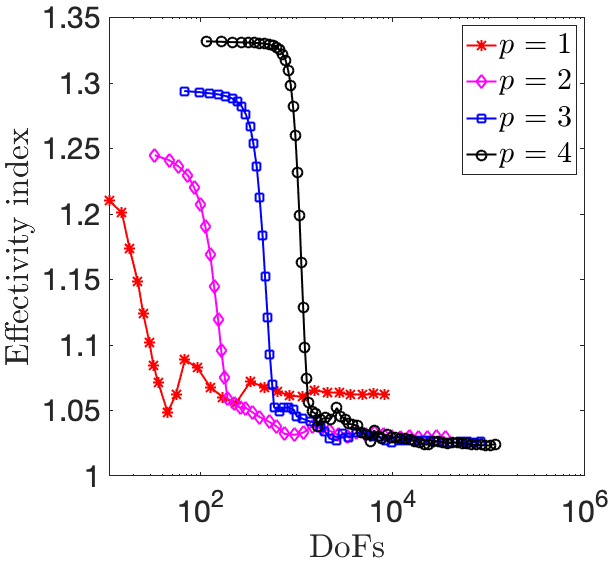}
\qquad
\includegraphics[width=0.45\linewidth]{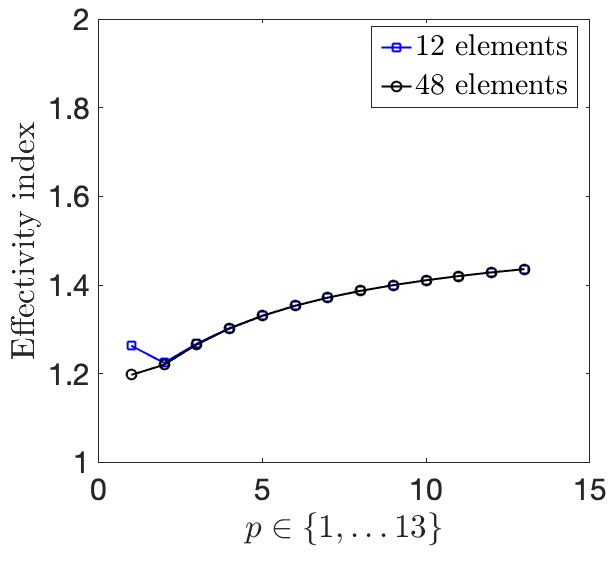}
\end{center}
\caption{Example 1. Effectivity index of estimator $\eta_\ell(\TT_\ell)$ as a function of DoFs for polynomial degrees $p\in\{1,2,3,4\}$ (left).
Effectivity index as a function of polynomial degrees $p\in\{1, \dots,13\}$ on a mesh composed of 12 and 48 triangles (right).}\label{Ex1:effectivity p=1,2,3,4}
\end{figure}

Next, we numerically test the value of the constants $\const{lb}(\bs{a}) = \eta_\ell(\bs{a})/\norm{\nabla r_{\ell+1}^{\bs{a}}}{\oma}$ defined in~\eqref{eq:Clb_1} for all marked vertices $\bs{a} \in\MM_\ell$.
In Figure~\ref{Ex1:C_{lb}=3,4}, we display the maximal and minimal values of $\const{lb}(\bs{a})$ over all $\bs{a}\in\MM_\ell$ for all $p\in \{1,\dots,4\}$.
The values lie between $0.2$ and $1.6$. Thus, the second condition in~\eqref{eq:stop_refine} is always satisfied in this test case. We also mention that only for $p=1$, the adaptive algorithm performs more than one newest vertex bisection $\beta>1$ to guarantee the second requirement from~\eqref{eq:stop_refine},
i.e., it is occasionally required to reach $\beta=2,3$, while $\beta=1,2$ can even lead to $\const{lb}(\bs{a}) = \infty$.

\begin{figure}[!tb]
\begin{center}
\begin{tabular}{cc}
\includegraphics[width=0.45\linewidth]{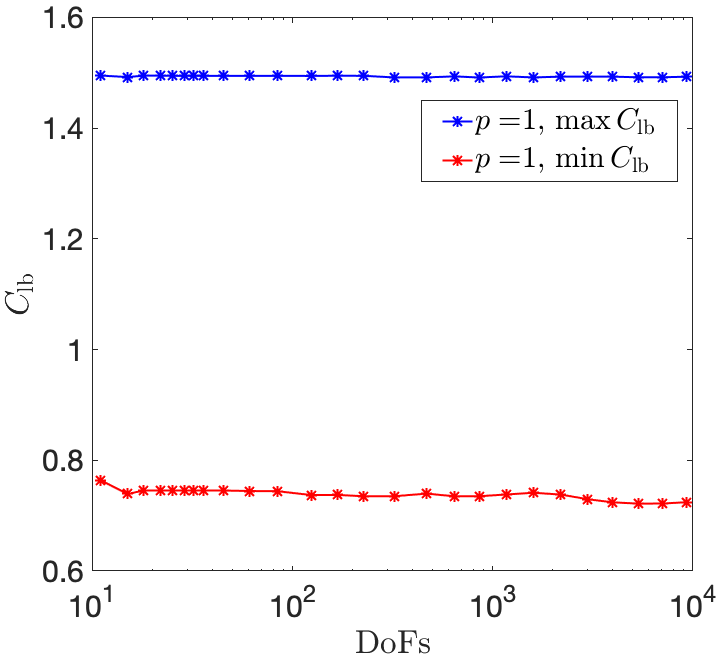}
&
\includegraphics[width=0.45\linewidth]{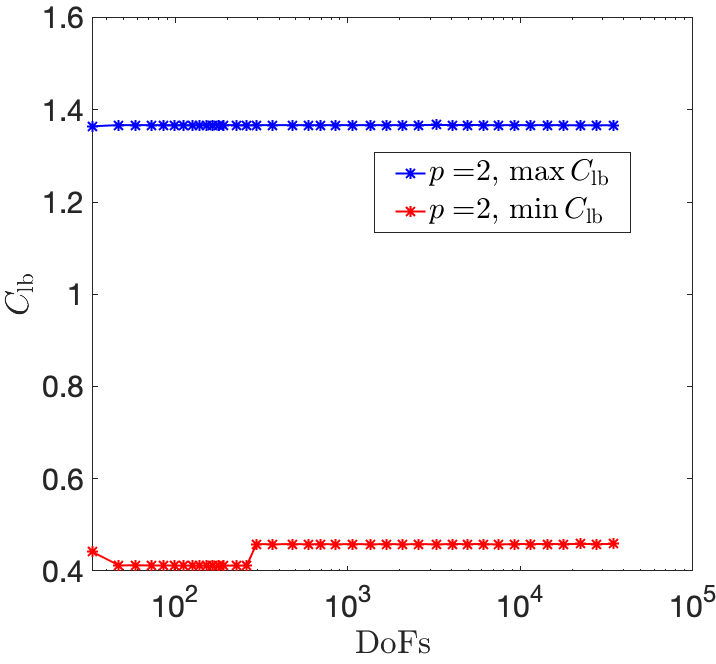} \\
\includegraphics[width=0.45\linewidth]{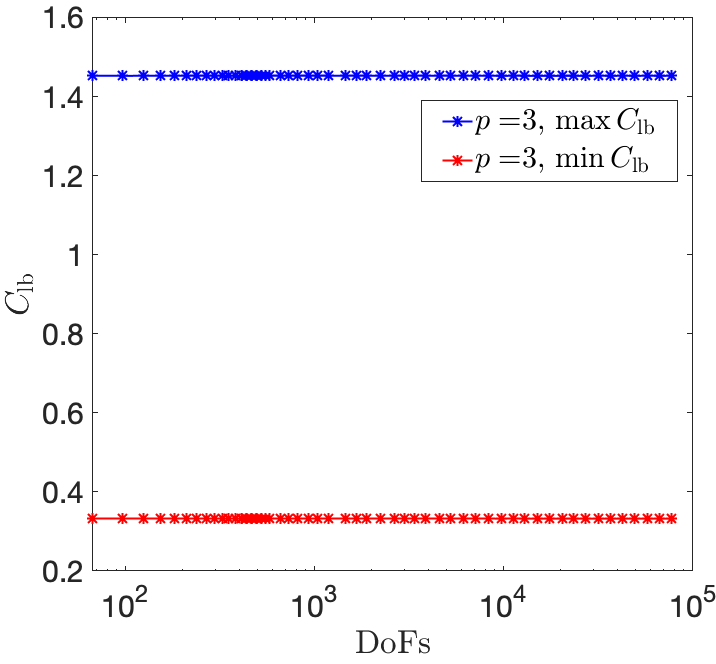}
&
\includegraphics[width=0.45\linewidth]{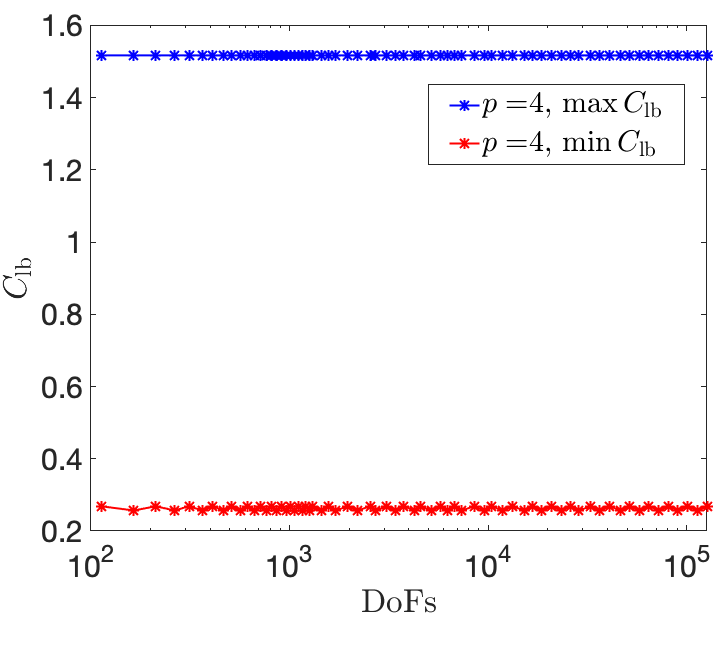}
\end{tabular}
\end{center}
\caption{Example 1. Maximal and minimal values over all marked vertices $\bs{a}\in\MM_\ell$ of the local stability constants $\const{lb}(\bs{a})$ from~\eqref{eq:Clb_1} for polynomial degrees $p\in \{1,2,3,4\}$.}\label{Ex1:C_{lb}=3,4}
\end{figure}

Finally, we numerically test the value of the effectivity indices of the estimated error reduction factor $q_{\rm ctr}$ defined in~\eqref{eq:error_contraction}, i.e., the ratio $q_{\rm ctr} / \big(\norm{\nabla (u - u_{\ell+1})}{\Omega}/\norm{\nabla (u - u_{\ell})}{\Omega}\big)$. In Figure~\ref{Ex1: effectivity for error reduction, p=1,2,3,4}, we display these ratios for all $p\in \{1,\dots,4\}$.
They lie between $1$ and $1.6$.

\begin{figure}[!tb]
\begin{center}
\includegraphics[width=0.45\linewidth]{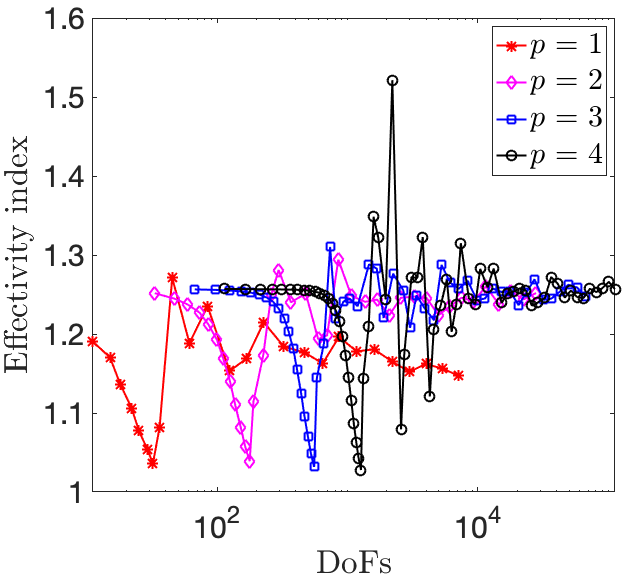}
\end{center}
\caption{Example 1. Effectivity index of error reduction factor $q_{\rm ctr}$ as a function of DoFs for polynomial degrees $p\in\{1,2,3,4\}$.}\label{Ex1: effectivity for error reduction, p=1,2,3,4}
\end{figure}

\subsection{Example 2: Adaptive FEM on cross-shape without unkown solution}

We select  homogeneous Dirichlet boundary conditions on the cross-shaped domain $\Omega := \set{(x_1,x_2)\in (-1,1)^2}{|x_1|\not\in [0.5,1] \text{ or }  |x_2| \not\in [0.5,1]}$ and choose the forcing term as
\begin{equation}
f :=1 .
\end{equation}
The analytical solution of this example is not known.

We first test the convergence rate of the above adaptive algorithm with polynomial degrees $p\in\{1,2,3,4\}$.
The {\sl a posteriori} error estimator  $\eta_\ell(\TT_\ell)$ is reported in Figure~\ref{Ex2:error_estimator_p=1,2,3,4} as a function of DoFs.
We observe that the estimator converges at the optimal rate $\mathcal{O}(\textup{DoFs}^{-\frac{p}{2}})$.
Moreover, we observe in Figure~\ref{Ex2:mesh p=1,2,3,4} that the adaptive algorithm constructs strongly graded meshes near the four reentrant corner points $(0.5,0.5)$, $(-0.5,0.5)$, $(0.5,-0.5)$, and $(-0.5,-0.5)$. The higher the polynomial degree $p$, the stronger the mesh grading becomes.

Next, we numerically test the value of the constants $\const{lb}(\bs{a}) = \eta_\ell(\bs{a})/\norm{\nabla r_{\ell+1}^{\bs{a}}}{\oma}$ defined in~\eqref{eq:Clb_1} for all marked vertices $\bs{a} \in\MM_\ell$.
In Figure~\ref{Ex2:C_{lb}=3,4}, we display the maximal and minimal values of $\const{lb}(\bs{a})$ over all $\bs{a}\in\MM_\ell$ for all $p\in \{1,\dots,4\}$.
The values lie between $0.2$ and $1.6$. Thus, the second condition in~\eqref{eq:stop_refine} is always satisfied in this test case.
We also mention that only for $p=1$, the adaptive algorithm performs more than one newest vertex bisection $\beta>1$ to guarantee the second requirement from~\eqref{eq:stop_refine},
i.e., it is occasionally required to reach $\beta=2,3$, while $\beta=1,2$ can even lead to $\const{lb}(\bs{a}) = \infty$.

\begin{figure}[!tb]
\begin{center}
\begin{tabular}{cc}
\includegraphics[width=0.46\linewidth]{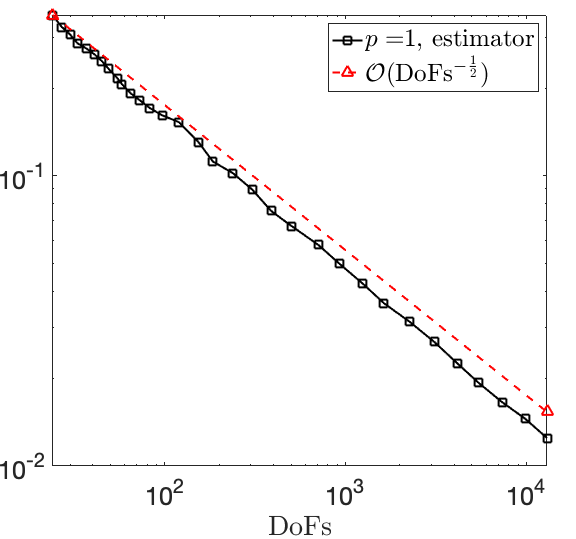} &
\includegraphics[width=0.46\linewidth]{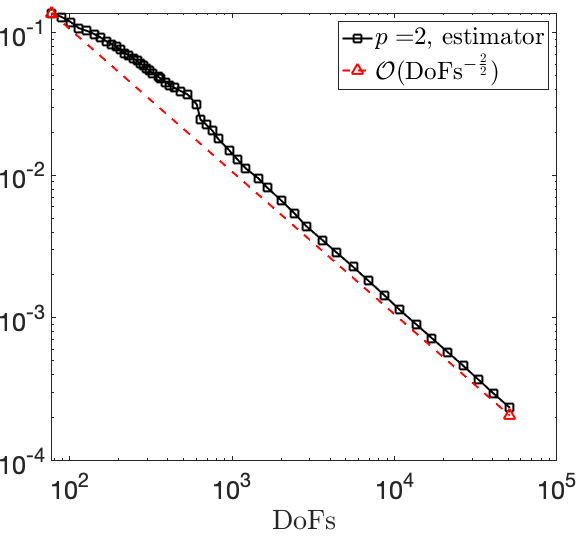} \\
\includegraphics[width=0.46\linewidth]{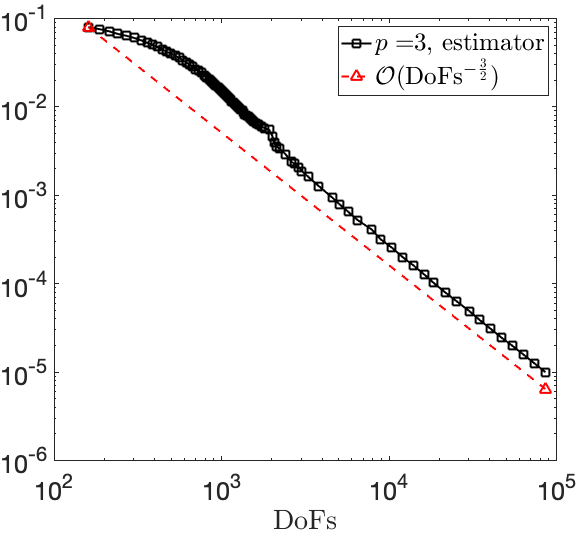} &
\includegraphics[width=0.46\linewidth]{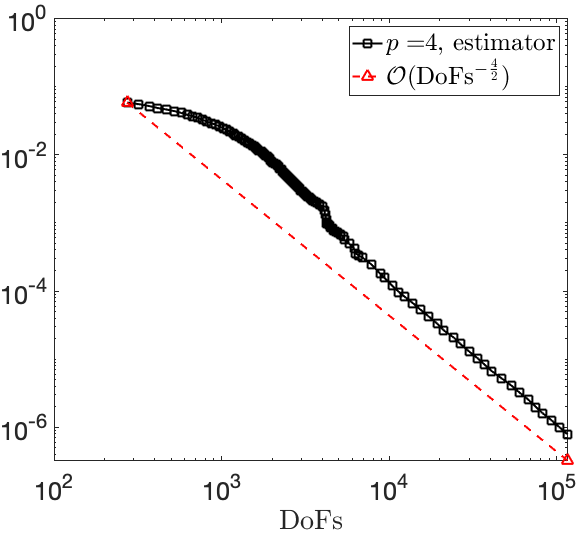}
\end{tabular}
\end{center}
\caption{Example 2. {\sl A posteriori} error estimator as a function of DoFs for polynomial degrees $p\in\{1,2,3,4\}$.}\label{Ex2:error_estimator_p=1,2,3,4}
\end{figure}

\begin{figure}[!tb]
\begin{center}
\begin{tabular}{cc}
\includegraphics[width=0.45\linewidth]{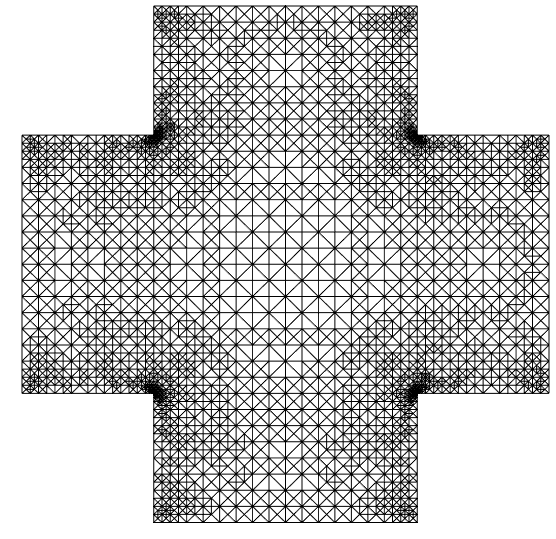}
&
\includegraphics[width=0.45\linewidth]{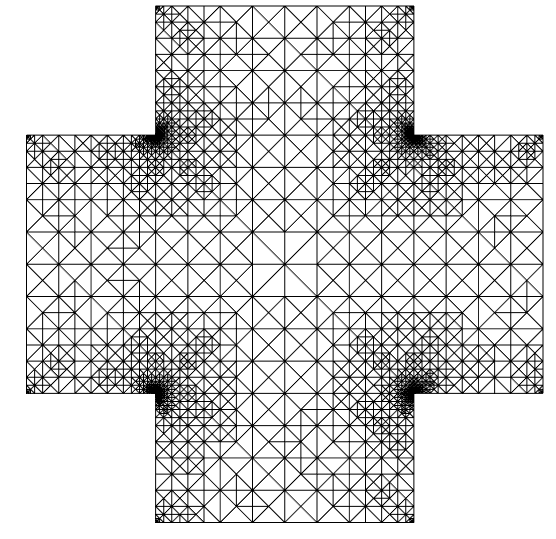}
\end{tabular}
\end{center}
\caption{Example 2. Adaptive mesh for polynomial degree $p=1$ with $5956$ elements (left) and for polynomial degree $p=4$ with $5720$ elements (right).}\label{Ex2:mesh p=1,2,3,4}
\end{figure}

\begin{figure}[!tb]
\begin{center}
\begin{tabular}{cc}
\includegraphics[width=0.45\linewidth]{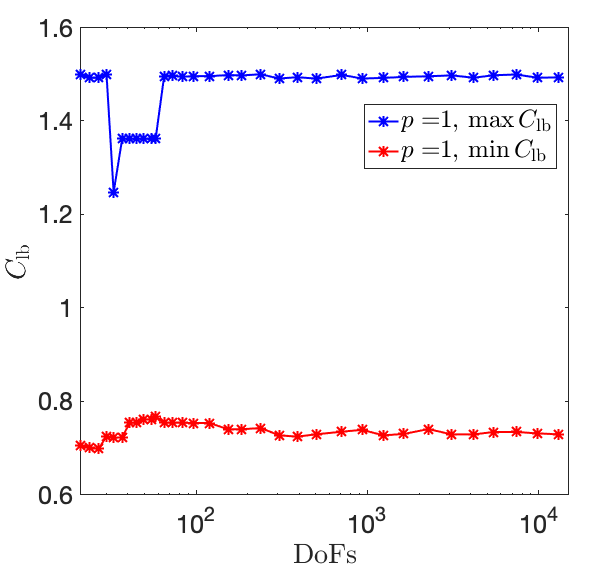}
&
\includegraphics[width=0.45\linewidth]{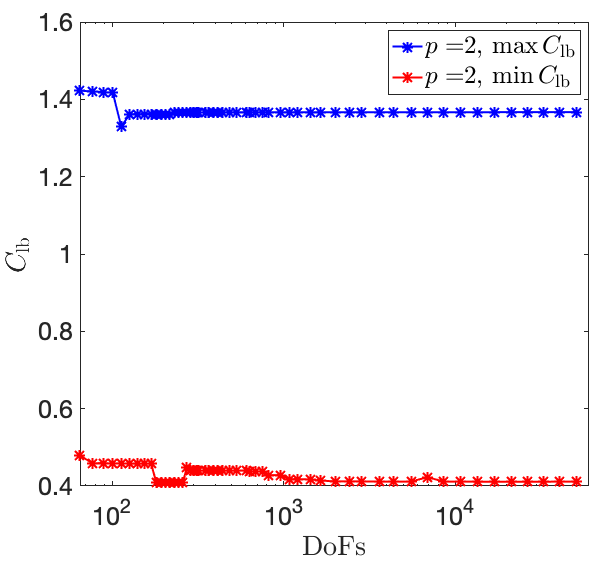} \\
\includegraphics[width=0.45\linewidth]{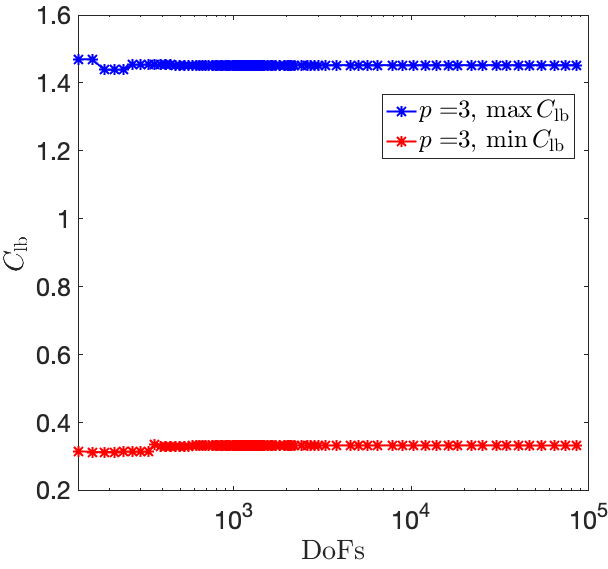}
&
\includegraphics[width=0.45\linewidth]{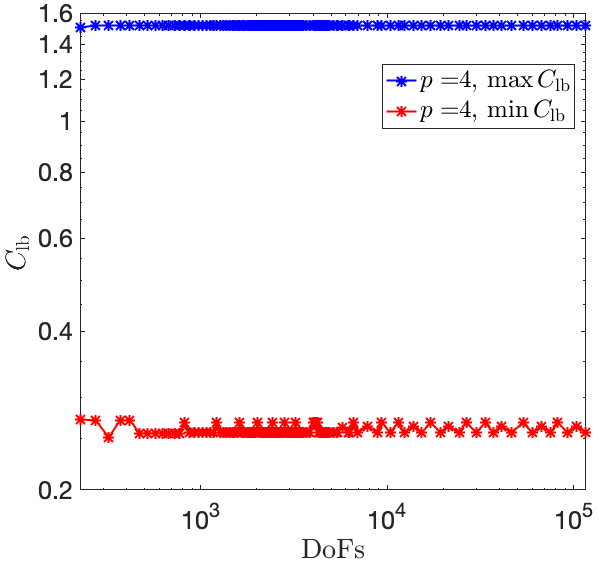}
\end{tabular}
\end{center}
\caption{Example 2. Maximal and minimal values over all marked vertices $\bs{a}\in\MM_\ell$ of the local stability constants $\const{lb}(\bs{a})$ from~\eqref{eq:Clb_1} for polynomial degrees $p\in \{1,2,3,4\}$.}\label{Ex2:C_{lb}=3,4}
\end{figure}

Finally, we numerically test the value of the effectivity indices of the estimated error reduction factor $q_{\rm ctr}$ defined in~\eqref{eq:error_contraction}. 
Since the exact solution and thus the error are unknown, we compute the ratio $q_{\rm ctr} / \big(\eta_{\ell+1}(\TT_{\ell+1}) / \eta_\ell(\TT_\ell)\big)$ instead. In Figure~\ref{Ex2: effectivity for error reduction, p=1,2,3,4}, we display these ratios for all $p\in \{1,\dots,4\}$.
They lie between $1$ and $1.4$.

\begin{figure}[!tb]
\begin{center}
\includegraphics[width=0.45\linewidth]{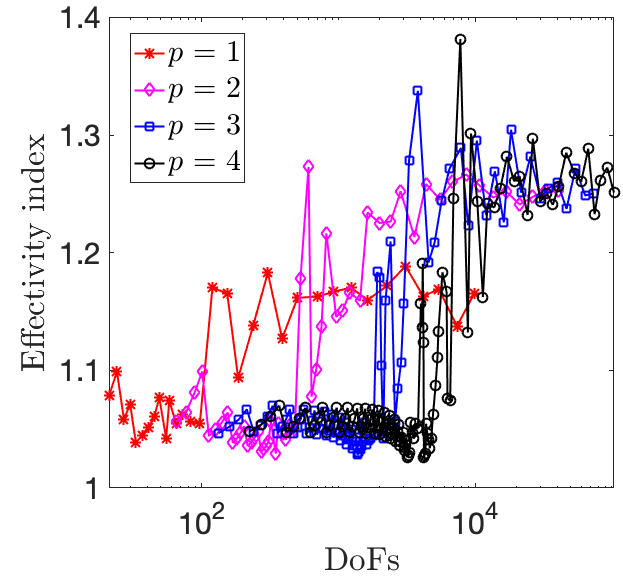}
\end{center}
\caption{Example 2. Effectivity index of error reduction factor $q_{\rm ctr}$ as a function of DoFs for polynomial degrees $p\in\{1,2,3,4\}$.}\label{Ex2: effectivity for error reduction, p=1,2,3,4}
\end{figure}

\appendix

\section{Local equivalences}\label{sec:equivalence}

We show in this appendix a local equivalence of the vertex- and element-based version of the equilibrated-flux estimator, as well as their local equivalence to weighted residual estimators.

\subsection{Local equivalences of equilibrated-flux estimators}\label{sec:equivalence_flux}

Estimates~\eqref{eq:triangle}--\eqref{eq:oscillations} from Section~\ref{sec:reliability} yield that
\begin{align}
	\eta_\ell(T) \le (d+1)^{1/2} \eta_\ell(\VV_T) \quad\text{for all } T\in\TT_\ell.
\end{align}
More precisely, we have that
\begin{align}\label{eq:el2ver}
	\eta_\ell(\mathcal{S}_\ell) \le (d+1)^{1/2} \eta_\ell\big(\set{\bs{a}\in\VV_\ell}{\bs{a}\in\VV_T, \, T \in \mathcal{S}_\ell}\big) \quad\text{for all } \mathcal{S}_\ell\subseteq\TT_\ell.
\end{align}

Also a converse estimate holds true, i.e.,
\begin{align}\label{eq:ver2el}
	\eta_\ell(\bs{a}) \lesssim \eta_\ell(\TT_\ell(\bs{a}))
	\quad \text{for all }\bs{a}\in\VV_\ell.
\end{align}
This follows from the crucial estimate
$\norm{\psi_\ell^{\bs{a}}\nabla u_\ell + \bs{\sigma}_\ell^{\bs{a}}}{\oma} \le \const{st} \norm{\nabla\rho^{\bs{a}}}{\oma}$ from~\eqref{eq:est2res}:
For all $v\in H_*^1(\oma)$, $\nabla{\cdot}\bs{\sigma}_\ell = \Pi_\ell^{p} f$ shows that
\begin{align}
	&\dual{\nabla \rho^{\bs{a}}}{\nabla v}_{\oma}
	= \dual{\Pi_\ell^{p}(\psi_\ell^{\bs{a}} f) - \nabla \psi_\ell^{\bs{a}} \cdot \nabla u_\ell}{v}_{\oma} - \dual{\psi_\ell^{\bs{a}} \nabla u_\ell}{\nabla v}_{\oma}
	\nonumber \\ &\quad
	= \dual{\psi_\ell^{\bs{a}} \nabla{\cdot} \bs{\sigma}_\ell - \nabla \psi_\ell^{\bs{a}} \cdot \nabla u_\ell}{v}_{\oma} - \dual{\psi_\ell^{\bs{a}} \nabla u_\ell}{\nabla v}_{\oma}
		+\dual{\Pi_\ell^{p}(\psi_\ell^{\bs{a}} f) - \psi_\ell^{\bs{a}} \Pi_\ell^{p} f}{v}_{\oma}
	\label{eq_eq}\\ &\quad
	= -\dual{\bs{\sigma}_\ell + \nabla u_\ell}{\nabla(\psi_\ell^{\bs{a}} v)}_{\oma} +\dual{\Pi_\ell^{p}(\psi_\ell^{\bs{a}} f) - \psi_\ell^{\bs{a}} \Pi_\ell^{p} f}{v}_{\oma}. \nonumber
\end{align}
If $f\in\mathcal{P}^{p-1}(\TT_\ell)$, the last term and the oscillation term in $\eta_\ell(\bs{a})$ vanish,
the desired inequality~\eqref{eq:ver2el} follows with a multiplicative constant depending only on the dimension $d$ and shape regularity of $\TT_\ell$, namely,
\begin{align}\label{eq:ver2el_2}
\begin{split}
	\eta_\ell(\bs{a}) & = \norm{\psi_\ell^{\bs{a}} \nabla u_\ell + \bs{\sigma}_\ell^{\bs{a}}}{\oma}
	\stackrel{\eqref{eq:est2res}}\le \const{st} \norm{\nabla\rho^{\bs{a}}}{\oma}
	\stackrel{\eqref{eq_eq}+\eqref{eq:PFS}} \le \const{st}\const{cont,PF} \norm{\nabla u_\ell + \bs{\sigma}_\ell}{\oma}
	\\
	&=  \const{st}\const{cont,PF}\eta_\ell(\TT_\ell(\bs{a})).
\end{split}
\end{align}
For the general case $f\not\in\mathcal{P}^{p-1}(\TT_\ell)$, one can employ the oscillation bound~\eqref{eq:Cosc}, albeit at the expense of a constant depending on the polynomial degree $p$.

\subsection{Local equivalence to weighted residual estimator}\label{sec:equi2res}

Existing optimality proofs of adaptive algorithms steered by equilibrated-flux estimators~\cite{ks11,cn12,Ber_Bof_Prag_Syng_a_post_20} crucially exploit local equivalence to the weighted residual estimator $\zeta_\ell(\TT_\ell)$, where
\begin{align}\label{eq:weighted-residual}
	\zeta_\ell(T)^2 := |T|^{2/d} \norm{f + \Delta u_\ell}{T}^2 + |T|^{1/d} \norm{\jump{\partial_{\bs{n}} u_\ell} }{\partial T\cap\Omega}^2,
\end{align}
and, as above,
\begin{align*}
	\zeta_\ell(\mathcal{S}_\ell) := \Big(\sum_{T\in\mathcal{S}_\ell} \zeta_\ell(T)^2\Big)^{1/2}
	\quad \text{for all }\mathcal{S}_\ell\subseteq\TT_\ell.
\end{align*}
As usual, $\jump{\partial_{\bs{n}} u_\ell}$ denotes the jump of the normal component of the weak gradient of $u_\ell$ across interior faces.
Let $\TT_\ell^*(T):=\set{T'\in\TT_\ell}{T' \text{ shares a face with }T}$ and $\TT_\ell(T):=\set{T'\in\TT_\ell}{T' \text{ shares a vertex with }T}$.
Then it holds that
\begin{align}\label{eq:equivalence_weighted}
	\zeta_\ell(T) \lesssim \eta_\ell(\TT_\ell^*(T))
	\quad\text{and}\quad
	\eta_\ell(T) \lesssim \zeta_\ell(\TT_\ell(T))
	\quad \text{for all }T\in\TT_\ell,
\end{align}
see~\cite{Ched_Fuc_Priet_Voh_guar_rob_FE_RD_09,ks11,cn12,Ern_Voh_adpt_IN_13,Ber_Bof_Prag_Syng_a_post_20}. Indeed, the first inequality of~\eqref{eq:equivalence_weighted} follows readily from the identities for the source $f = \nabla \cdot\bs{\sigma}_\ell + (1-\Pi_\ell^{p}) f$
and the normal jump $\jump{\bs{\sigma}_\ell\cdot \bs{n}}|_{\partial T\cap\Omega} = 0$ (due to $\bs{\sigma}_\ell\in \bs{H}(\div;\Omega)$) in combination with inverse and trace inequalities.
The involved constant thus depends on the dimension $d$, the shape regularity of $\TT_\ell$, and the polynomial degree $p$.
The second inequality in the above references follows from equivalence of norms on finite-dimensional spaces, which again implies a $p$-dependent constant.
We now establish it with a constant that depends only on the dimension $d$ and the shape regularity of $\TT_\ell$.

We exploit~\eqref{eq:est2res}, i.e.,
\begin{align*}
	\norm{\nabla u_\ell + \bs{\sigma}_\ell}{T}
	\le \sum_{\bs{a}\in \VV_T} \norm{\psi_\ell^{\bs{a}}\nabla u_\ell + \bs{\sigma}_\ell^{\bs{a}}}{T}
	\le \const{st} \sum_{\bs{a}\in\VV_T} \norm{\nabla\rho^{\bs{a}}}{\oma}.
\end{align*}
For $v \in H_*^1(\oma)$, integration by parts gives that
\begin{align*}
	\dual{\nabla \rho^{\bs{a}}}{\nabla v}_{\oma}
	= {} & \dual{(\Pi_\ell^{p}-1)(\psi_\ell^{\bs{a}} f)}{v}_{\oma}
		+ \sum_{T\in\TT_\ell(\bs{a})} \dual{f + \Delta u_\ell}{\psi_\ell^{\bs{a}} v}_{T} \\
		{} & - \frac{1}{2}\dual{\jump{\partial_{\bs{n}} u_\ell}}{\psi_\ell^{\bs{a}} v}_{\partial T \setminus \partial\oma}.
\end{align*}
The Poincar\'e--Friedrichs inequality on vertex patches~\eqref{eq:PF} and the trace inequality imply that
\begin{align*}
	\norm{\nabla \rho^{\bs{a}}}{\oma}
	\lesssim \zeta_\ell(\TT_\ell(\bs{a})) + \diam(\oma) \norm{\psi_\ell^{\bs{a}} f - \Pi_\ell^{p}(\psi_\ell^{\bs{a}} f)}{\oma}.
\end{align*}
For $T\in\TT_\ell^{\bs{a}}$, it holds that
\begin{align*}
	\norm{\psi_\ell^{\bs{a}} f - \Pi_\ell^{p}(\psi_\ell^{\bs{a}} f)}{T}
	= \norm{(1-\Pi_\ell^{p}) [\psi_\ell^{\bs{a}} (f + \Delta u_\ell)])}{T}
	\le \norm{f + \Delta u_\ell}{T}.
\end{align*}
The oscillation term $\norm{f-\Pi_\ell^{p} f}{T}$ in $\eta_\ell(T)$ can be treated as in the last inequality.
Overall, this gives the desired second inequality of~\eqref{eq:equivalence_weighted}.

\begin{remark}[oscillation bound] \label{rem:Cosc}
The arguments of this section also readily show the oscillation bound~\eqref{eq:Cosc}.
\end{remark}

\section{Element-based adaptive algorithm \& optimal convergence}\label{sec:algorithm_element}

We finally consider the following counterpart of Algorithm~\ref{alg:afem}, where standard element-based D\"orfler marking is used and no additional newest-vertex bisections are applied.

\begin{algorithm}[element-based, standard]\label{alg:afem_element}
\textbf{Input:} Triangulation $\TT_0$, D\"orfler marking parameter $0 < \theta \le 1$.
\\
\textbf{Loop:} For each $\ell=0,1,2,\dots$, iterate the following steps {\rm (i)--(iv)}:
\begin{enumerate}[\rm(i)]
\item Compute the Galerkin approximation $u_\ell\in\mathbb{V}_\ell$.
\item Compute the error indicators $\eta_\ell(T)$ from~\eqref{eq:reliability_one} for all elements $T\in\TT_\ell$.
\item Determine a minimal set of marked elements $\MM_\ell\subseteq\TT_\ell$ satisfying the D\"orfler marking
\begin{align}\label{eq:doerfler_element}
	\theta\,\eta_\ell(\TT_\ell) \le \eta_\ell(\MM_\ell).
\end{align}
\item Generate refined triangulation $\TT_{\ell+1}:=\refine(\TT_\ell,\MM_\ell)$, employing one newest-vertex bisection to all $T\in\MM_\ell$ and a minimal number of additional newest-vertex bisections to ensure conformity of the resulting triangulation $\TT_{\ell+1}$.
\end{enumerate}
\textbf{Output:} Nested sequence of triangulations $\TT_\ell$, corresponding Galerkin approximations $u_\ell \in \mathbb{V}_\ell$, and
equilibrated-flux estimators $\eta_\ell(\TT_\ell)$ for all $\ell \in \N_0$.
\end{algorithm}

Building on the ($p$-dependent) local estimator equivalence of Section~\ref{sec:equi2res}, linear convergence~\eqref{eq:linear_convergence} at optimal rate~\eqref{eq:optimal_convergence} of this element-based algorithm is well-known~\cite{ks11,cn12,Ber_Bof_Prag_Syng_a_post_20};
see also the general framework~\cite[Section~8]{cfpp14}.
While the standard argument particularly leads to a $p$-dependent optimal D\"orfler marking parameter $\tilde\theta_{\rm opt}$, our novel Lemma~\ref{lem:optimality_doerfler} allows for a $p$-robust $\tilde\theta_{\rm opt}$ --- at least if $f\in\mathcal{P}^{p-1}(\TT_0)$. In contrast to Theorem~\ref{thm:optimal_convergence}, the constant $\const{opt}$ here always depends on the polynomial degree $p$.

\begin{theorem}[optimal convergence]
Let $\tilde\theta_{\rm opt}$ be the optimal D\"orfler marking parameter of Lemma~\ref{lem:optimality_doerfler}, which is $p$-robust if $f\in\mathcal{P}^{p-1}(\TT_0)$.
Then, for all $0<\theta<\tilde\theta_{\rm opt}$ and all $s>0$, there exist constants $c_{\rm opt},\const{opt}>0$ such that the meshes $(\TT_\ell)_{\ell\in\N_0}$ generated by Algorithm~\ref{alg:afem_element} satisfy that
\begin{align}\label{eq:optimal_convergence2}
	c_{\rm opt} \norm{u}{\A_s^{\rm tot}}
	\le\sup_{\ell\in\N_0}\big[[p^{d}(\#\TT_\ell-\#\TT_0+1)]^{s} \big(\norm{\nabla(u - u_\ell)}{\Omega} + \osc_\ell(\TT_\ell)\big)\big]
	\le \const{opt} \norm{u}{\A_s^{\rm tot}},
\end{align}
where $\norm{u}{\A_s^{\rm tot}}$ is defined as in~\eqref{eq:approximation} for the total error $\norm{\nabla(u - u_\star)}{\Omega}+\osc_\star(\TT_\star)$.
The constant $\const{opt}$ depends  only on the dimension $d$, shape regularity of $\TT_0$, the constant $\const{clos}$ from~\eqref{eq:closure}, the polynomial degree $p$, the D\"orfler marking parameter $\theta$, and on $s$,
while $c_{\rm opt}$ depends only on $ \#\TT_0$, the constant $\const{child}$ from~\eqref{eq:child},  $s$, and, if there exists $\ell_0$ with $\eta_{\ell_0}=0$, also on $\ell_0$.
\end{theorem}

\begin{proof}
As in Remark~\ref{rem:standard_proof}, the local estimator equivalence of Section~\ref{sec:equi2res} implies linear convergence~\eqref{eq:linear_convergence} of the equilibrated-flux estimator $\eta_\ell(\TT_\ell)$;
see again \cite[Section~8]{cfpp14} or \cite{ks11,cn12,Ber_Bof_Prag_Syng_a_post_20} for the standard argument.
Then, the statement follows from optimality of D\"orfler marking (Lemma~\ref{lem:optimality_doerfler}), the comparison lemma (Lemma~\ref{lem:comparison}) for the estimator $\eta_\star(\TT_\star)$ instead of the error $\norm{\nabla (u - u_\star)}{\Omega}$, and the equivalence of the estimator $\eta_\star(\TT_\star)$ to the total error $\norm{\nabla(u - u_\star)}{\Omega}+\osc_\star(\TT_\star)$ from~\eqref{eq:releff}; see, e.g., \cite[Proposition~4.15]{cfpp14} or the proof of Theorem~\ref{thm:optimal_convergence}.
\end{proof}

\bibliographystyle{alpha}
\bibliography{literature}

@article{bcnv24,
	author = {Bonito, Andrea and Canuto, Claudio and Nochetto, Ricardo H. and Veeser, Andreas},
	doi = {10.1017/S0962492924000011},
	fjournal = {Acta Numerica},
	issn = {0962-4929},
	journal = {Acta Numer.},
	mrclass = {65N30 (65N35 65N50)},
	mrnumber = {4793681},
	pages = {163--485},
	title = {Adaptive finite element methods},
	url = {https://doi.org/10.1017/S0962492924000011},
	volume = {33},
	year = {2024}}

@article{DolejsiErnVohralik16,
	author = {Dolej\v{s}\'{i}, V\'{i}t and Ern, Alexandre and Vohral\'{i}k, Martin},
	doi = {10.1137/15M1026687},
	fjournal = {SIAM Journal on Scientific Computing},
	issn = {1064-8275,1095-7197},
	journal = {SIAM J. Sci. Comput.},
	mrclass = {65N30 (65N15 65N50)},
	mrnumber = {3556071},
	mrreviewer = {Rommel\ Bustinza},
	number = {5},
	pages = {A3220--A3246},
	title = {{$hp$}-adaptation driven by polynomial-degree-robust a posteriori error estimates for elliptic problems},
	url = {https://doi.org/10.1137/15M1026687},
	volume = {38},
	year = {2016}}

@article{DemkowiczBuffa05,
	author = {Demkowicz, Leszek and Buffa, Annalisa},
	doi = {10.1016/j.cma.2004.07.007},
	fjournal = {Computer Methods in Applied Mechanics and Engineering},
	issn = {0045-7825,1879-2138},
	journal = {Comput. Methods Appl. Mech. Engrg.},
	mrclass = {65N30 (78M10)},
	mrnumber = {2105164},
	mrreviewer = {JiChun\ Li},
	number = {2-5},
	pages = {267--296},
	title = {{$H^1$}, {$H({\rm curl})$} and {$H({\rm div})$}-conforming projection-based interpolation in three dimensions. {Q}uasi-optimal {$p$}-interpolation estimates},
	url = {https://doi.org/10.1016/j.cma.2004.07.007},
	volume = {194},
	year = {2005}}

@article{bdfgps26,
	author = {Philipp Bringmann and Aleksandar Dadic and Dario Ferloni and Gregor Gantner and Dirk Praetorius and Julian Streitberger},
	journal = {arXiv preprint},
	title = {{Optimal complexity of adaptive FEM for second-order linear elliptic PDEs driven by non-residual estimators, Part I: Symmetric PDEs}},
	volume = {2607.14820},
	year = {2026}}

@article{dk21,
	author = {Diening, Lars and Kreuzer, Christian},
	doi = {10.1515/cmam-2020-0041},
	fjournal = {Computational Methods in Applied Mathematics},
	issn = {1609-4840},
	journal = {Comput. Methods Appl. Math.},
	mrclass = {65N30 (65N15 65N50)},
	mrnumber = {4279090},
	number = {3},
	pages = {557--567},
	title = {On the threshold condition for {D}\"{o}rfler marking},
	url = {https://doi.org/10.1515/cmam-2020-0041},
	volume = {21},
	year = {2021}}

@article{gm09,
	author = {Gaspoz, Fernando D. and Morin, Pedro},
	doi = {10.1093/imanum/drn039},
	fjournal = {IMA Journal of Numerical Analysis},
	issn = {0272-4979},
	journal = {IMA J. Numer. Anal.},
	mrclass = {65N30 (65N12 65N50)},
	mrnumber = {2557050},
	mrreviewer = {Christos A. Xenophontos},
	number = {4},
	pages = {917--936},
	title = {Convergence rates for adaptive finite elements},
	url = {https://doi.org/10.1093/imanum/drn039},
	volume = {29},
	year = {2009}}

@article{gss14,
	author = {Gallistl, Dietmar and Schedensack, Mira and Stevenson, Rob P.},
	journal = {Comput. Methods Appl. Math.},
	number = {3},
	pages = {317--320},
	publisher = {De Gruyter},
	title = {A remark on newest vertex bisection in any space dimension},
	volume = {14},
	year = {2014}}

@article{dgs25,
	author = {Diening, Lars and Gehring, Lukas and Storn, Johannes},
	journal = {Found. Comput. Math.},
	note = {DOI~10.1007/s10208-025-09698-7},
	publisher = {Springer},
	title = {Adaptive mesh refinement for arbitrary initial triangulations},
	year = {2025}}

@article{cv26,
	author = {Chaumont-Frelet, Th{\'e}ophile and Vohral{\'{\i}}k, Martin},
	doi = {10.1007/s10208-024-09674-7},
	fjournal = {Foundations of Computational Mathematics. The Journal of the Society for the Foundations of Computational Mathematics},
	journal = {Found. Comput. Math.},
	number = {1},
	pages = {483--524},
	title = {Constrained and unconstrained stable discrete minimizations for $p$-robust local reconstructions in vertex patches in the de {R}ham complex},
	url = {https://doi.org/10.1007/s10208-024-09674-7},
	volume = {26},
	year = {2026}}

@article{cf00,
	author = {Carstensen, Carsten and Funken, Stefan A.},
	journal = {East-West J. Numer. Math.},
	number = {3},
	pages = {153--175},
	title = {{Constants in Cl{\'e}ment-interpolation error and residual based a posteriori estimates in finite element methods}},
	volume = {8},
	year = {2000}}

@article{vohralik26,
	author = {Vohral{\'{\i}}k, Martin},
	fjournal = {SIAM Journal on Numerical Analysis},
	journal = {SIAM J. Numer. Anal.},
	note = {Accepted for publication},
	title = {$p$-robust equivalence of global continuous and local discontinuous approximation, a $p$-stable local projector, and optimal elementwise $hp$ approximation estimates in {$H^1$}},
	url = {https://hal.inria.fr/hal-04436063},
	year = {2026}}

@article{pp20,
	author = {Pfeiler, Carl-Martin and Praetorius, Dirk},
	doi = {10.1090/mcom/3553},
	journal = {Math. Comp.},
	number = {326},
	pages = {2735--2752},
	title = {D{\"o}rfler marking with minimal cardinality is a linear complexity problem},
	volume = {89},
	year = {2020}}

@article{cnsv17b,
	author = {Canuto, Claudio and Nochetto, Ricardo H. and Stevenson, Rob and Verani, Marco},
	journal = {Comput. Math. Appl.},
	number = {9},
	pages = {2004--2022},
	publisher = {Elsevier},
	title = {On $p$-robust saturation for $hp$-{AFEM}},
	volume = {73},
	year = {2017}}

@article{cnsv17a,
	author = {Canuto, Claudio and Nochetto, Ricardo H. and Stevenson, Rob and Verani, Marco},
	journal = {Numer. Math.},
	number = {4},
	pages = {1073--1119},
	publisher = {Springer},
	title = {Convergence and optimality of $hp$-{AFEM}},
	volume = {135},
	year = {2017}}

@article{dk08,
	author = {Diening, Lars and Kreuzer, Christian},
	journal = {SIAM J. Numer. Anal.},
	number = {2},
	pages = {614--638},
	publisher = {SIAM},
	title = {Linear convergence of an adaptive finite element method for the $p$-{L}aplacian equation},
	volume = {46},
	year = {2008}}

@article{dev20,
	author = {Daniel, Patrik and Ern, Alexandre and Vohral{\'\i}k, Martin},
	doi = {10.1016/j.cma.2019.112607},
	journal = {Comput. Methods Appl. Mech. Engrg.},
	pages = {112607},
	publisher = {Elsevier},
	title = {An adaptive $hp$-refinement strategy with inexact solvers and computable guaranteed bound on the error reduction factor},
	volume = {359},
	year = {2020}}

@article{desv18,
	author = {Daniel, Patrik and Ern, Alexandre and Smears, Iain and Vohral{\'\i}k, Martin},
	doi = {10.1016/j.cma.2019.112607},
	journal = {Comput. Math. Appl.},
	number = {5},
	pages = {967--983},
	publisher = {Elsevier},
	title = {An adaptive $hp$-refinement strategy with computable guaranteed bound on the error reduction factor},
	volume = {76},
	year = {2018}}

@article{ev15,
	author = {Ern, Alexandre and Vohral{\'{\i}}k, Martin},
	doi = {10.1137/130950100},
	fjournal = {SIAM Journal on Numerical Analysis},
	issn = {0036-1429},
	journal = {SIAM J. Numer. Anal.},
	mrclass = {65N30 (65N15)},
	mrnumber = {3335498},
	number = {2},
	pages = {1058--1081},
	title = {Polynomial-degree-robust a posteriori estimates in a unified setting for conforming, nonconforming, discontinuous {G}alerkin, and mixed discretizations},
	url = {http://dx.doi.org/10.1137/130950100},
	volume = {53},
	year = {2015}}

@article{bps09,
	author = {Braess, Dietrich and Pillwein, Veronika and Sch{\"o}berl, Joachim},
	coden = {CMMECC},
	doi = {10.1016/j.cma.2008.12.010},
	fjournal = {Computer Methods in Applied Mechanics and Engineering},
	issn = {0045-7825},
	journal = {Comput. Methods Appl. Mech. Engrg.},
	mrclass = {65N15 (65N30)},
	mrnumber = {2500243},
	mrreviewer = {Mohammad Asadzadeh},
	number = {13-14},
	pages = {1189--1197},
	title = {Equilibrated residual error estimates are {$p$}-robust},
	url = {http://dx.doi.org/10.1016/j.cma.2008.12.010},
	volume = {198},
	year = {2009}}

@article{ps47,
	author = {Prager, William and Synge, John L.},
	doi = {10.1090/qam/25902},
	fjournal = {Quarterly of Applied Mathematics},
	issn = {0033-569X},
	journal = {Quart. Appl. Math.},
	mrclass = {73.2X},
	mrnumber = {MR0025902 (10,81b)},
	mrreviewer = {E. Reissner},
	pages = {241--269},
	title = {Approximations in elasticity based on the concept of function space},
	volume = {5},
	year = {1947}}

@book{repin08,
	address = {Berlin},
	author = {Repin, Sergey},
	doi = {10.1515/9783110203042},
	publisher = {De Gruyter},
	title = {A posteriori estimates for partial differential equations},
	year = {2008}}

@article{egp20,
	author = {Erath, Christoph and Gantner, Gregor and Praetorius, Dirk},
	doi = {10.1016/j.camwa.2019.07.014},
	journal = {Comput. Math. Appl.},
	number = {3},
	pages = {623--642},
	title = {Optimal convergence behavior of adaptive {FEM} driven by simple $(h- h/ 2)$-type error estimators},
	volume = {79},
	year = {2020}}

@article{dv23,
	author = {Patrik Daniel and Martin Vohral\'ik},
	doi = {10.1051/m2an/2022082},
	journal = {ESAIM Math. Model. Numer. Anal.},
	number = {1},
	pages = {329--366},
	publisher = {{EDP} Sciences},
	title = {Guaranteed contraction of adaptive inexact $hp$-refinement strategies with realistic stopping criteria},
	url = {https://doi.org/10.1051%2Fm2an%2F2022082},
	volume = {57},
	year = {2023}}

@article{traxler97,
	author = {Traxler, Christoph T.},
	journal = {Computing},
	number = {2},
	pages = {115--137},
	publisher = {Springer},
	title = {An algorithm for adaptive mesh refinement in $n$ dimensions},
	volume = {59},
	year = {1997}}

@article{maubach95,
	author = {Maubach, Joseph M.},
	journal = {{SIAM J. Sci. Comput.}},
	number = {1},
	pages = {210--227},
	publisher = {SIAM},
	title = {Local bisection refinement for $n$-simplicial grids generated by reflection},
	volume = {16},
	year = {1995}}

@article{ks11,
	author = {Kreuzer, Christian and Siebert, Kunibert G.},
	doi = {10.1007/s00211-010-0324-5},
	journal = {Numer. Math.},
	number = {4},
	pages = {679--716},
	publisher = {Springer},
	title = {{Decay rates of adaptive finite elements with D{\"o}rfler marking}},
	volume = {117},
	year = {2011}}

@article{cn12,
	author = {Casc{\'o}n, J. Manuel and Nochetto, Ricardo H.},
	doi = {10.1093/imanum/drr014},
	journal = {IMA J. Numer. Anal.},
	number = {1},
	pages = {1--29},
	publisher = {Oxford University Press},
	title = {Quasioptimal cardinality of {AFEM} driven by nonresidual estimators},
	volume = {32},
	year = {2012}}

@article{ev20,
	author = {Ern, Alexandre and Vohral{\'\i}k, Martin},
	doi = {10.1090/mcom/3482},
	journal = {Math. Comp.},
	number = {322},
	pages = {551--594},
	title = {Stable broken {$H^1$} and $\boldsymbol{H}(\mathrm{div})$ polynomial extensions for polynomial-degree-robust potential and flux reconstruction in three space dimensions},
	volume = {89},
	year = {2020}}

@article{bs08,
	author = {Braess, Dietrich and Sch{\"o}berl, Joachim},
	doi = {10.1090/S0025-5718-07-02080-7},
	journal = {Math. Comp.},
	number = {262},
	pages = {651--672},
	title = {Equilibrated residual error estimator for edge elements},
	volume = {77},
	year = {2008}}

@article{ckns08,
	author = {Casc\'{o}n, J. Manuel and Kreuzer, Christian and Nochetto, Ricardo H. and Siebert, Kunibert G.},
	doi = {10.1137/07069047X},
	journal = {SIAM J. Numer. Anal.},
	number = {5},
	pages = {2524--2550},
	publisher = {SIAM},
	title = {Quasi-optimal convergence rate for an adaptive finite element method},
	volume = {46},
	year = {2008}}

@article{stevenson07,
	author = {Stevenson, Rob},
	doi = {10.1007/s10208-005-0183-0},
	fjournal = {Foundations of Computational Mathematics. The Journal of the Society for the Foundations of Computational Mathematics},
	issn = {1615-3375},
	journal = {Found. Comput. Math.},
	mrclass = {65N30},
	mrnumber = {2324418},
	mrreviewer = {Erwin Stein},
	number = {2},
	pages = {245--269},
	title = {Optimality of a standard adaptive finite element method},
	volume = {7},
	year = {2007}}

@article{cfpp14,
	author = {Carstensen, Carsten and Feischl, Michael and Page, Marcus and Praetorius, Dirk},
	doi = {10.1016/j.camwa.2013.12.003},
	journal = {Comput. Math. Appl.},
	number = {6},
	pages = {1195--1253},
	publisher = {Elsevier},
	title = {Axioms of adaptivity},
	volume = {67},
	year = {2014}}

@article{bdd04,
	author = {Binev, Peter and Dahmen, Wolfgang and DeVore, Ron},
	doi = {10.1007/s00211-003-0492-7},
	journal = {Numer. Math.},
	number = {2},
	pages = {219--268},
	publisher = {Springer},
	title = {Adaptive finite element methods with convergence rates},
	volume = {97},
	year = {2004}}

@book{ao00,
	address = {New York},
	author = {Ainsworth, Mark and Oden, J. Tinsley},
	doi = {10.1002/9781118032824},
	publisher = {John Wiley \& Sons},
	title = {A posteriori error estimation in finite element analysis},
	year = {2000}}

@article{stevenson08,
	author = {Stevenson, Rob},
	doi = {10.1090/S0025-5718-07-01959-X},
	journal = {Math. Comp.},
	number = {261},
	pages = {227--241},
	title = {The completion of locally refined simplicial partitions created by bisection},
	volume = {77},
	year = {2008}}

@book{verfuerth13,
	author = {Verf{\"u}rth, R{\"u}diger},
	doi = {10.1093/acprof:oso/9780199679423.001.0001},
	isbn = {978-0-19-967942-3},
	mrclass = {65N30 (35J25 35K20 35Q30 35Q74 65N15)},
	mrnumber = {3059294},
	mrreviewer = {Manfred Dobrowolski},
	pages = {xx+393},
	publisher = {Oxford University Press, Oxford},
	title = {A posteriori error estimation techniques for finite element methods},
	url = {http://dx.doi.org/10.1093/acprof:oso/9780199679423.001.0001},
	year = {2013}}

@article{mns00,
	author = {Morin, Pedro and Nochetto, Ricardo H. and Siebert, Kunibert G.},
	doi = {10.1137/S0036142999360044},
	fjournal = {SIAM Journal on Numerical Analysis},
	issn = {0036-1429},
	journal = {SIAM J. Numer. Anal.},
	mrclass = {65N30 (65N55 65Y20)},
	mrnumber = {1770058},
	mrreviewer = {Petr N. Vabishchevich},
	number = {2},
	pages = {466--488},
	title = {Data oscillation and convergence of adaptive {FEM}},
	url = {http://dx.doi.org/10.1137/S0036142999360044},
	volume = {38},
	year = {2000}}

@article{doerfler96,
	author = {D{\"o}rfler, Willy},
	coden = {SJNAAM},
	doi = {10.1137/0733054},
	fjournal = {SIAM Journal on Numerical Analysis},
	issn = {0036-1429},
	journal = {SIAM J. Numer. Anal.},
	mrclass = {65N50 (65N55)},
	mrnumber = {1393904},
	mrreviewer = {S. F. McCormick},
	number = {3},
	pages = {1106--1124},
	title = {A convergent adaptive algorithm for {P}oisson's equation},
	url = {http://dx.doi.org/10.1137/0733054},
	volume = {33},
	year = {1996}}

@article{kpp13,
	author = {Karkulik, Michael and Pavlicek, David and Praetorius, Dirk},
	doi = {10.1007/s00365-013-9192-4},
	fjournal = {Constructive Approximation. An International Journal for Approximations and Expansions},
	issn = {0176-4276},
	journal = {Constr. Approx.},
	mrclass = {65N50 (65N30 65Y20)},
	mrnumber = {3097045},
	number = {2},
	pages = {213--234},
	title = {On 2{D} newest vertex bisection: optimality of mesh-closure and {$H^1$}-stability of {$L_2$}-projection},
	url = {http://dx.doi.org/10.1007/s00365-013-9192-4},
	volume = {38},
	year = {2013}}

@article{Ern_Voh_adpt_IN_13,
	author = {Ern, Alexandre and Vohral{\'{\i}}k, Martin},
	coden = {SJOCE3},
	doi = {10.1137/120896918},
	fjournal = {SIAM Journal on Scientific Computing},
	issn = {1064-8275},
	journal = {SIAM J. Sci. Comput.},
	mrclass = {Preliminary Data},
	mrnumber = {3072765},
	number = {4},
	pages = {A1761--A1791},
	title = {Adaptive inexact {N}ewton methods with a posteriori stopping criteria for nonlinear diffusion {PDE}s},
	url = {https://doi.org/10.1137/120896918},
	volume = {35},
	year = {2013}}

@incollection{Ber_Bof_Prag_Syng_a_post_20,
	author = {Bertrand, Fleurianne and Boffi, Daniele},
	booktitle = {75 years of mathematics of computation},
	doi = {10.1090/conm/754/15152},
	isbn = {978-1-4704-5163-9},
	mrclass = {65N30 (65N15 65N50)},
	mrnumber = {4132116},
	mrreviewer = {Huipo\ Liu},
	pages = {45--67},
	publisher = {Amer. Math. Soc., Providence, RI},
	series = {Contemp. Math.},
	title = {The {P}rager-{S}ynge theorem in reconstruction based a posteriori error estimation},
	url = {https://doi.org/10.1090/conm/754/15152},
	volume = {754},
	year = {2020}}

@article{Dest_Met_expl_err_CFE_99,
	author = {Destuynder, Philippe and M{\'e}tivet, Brigitte},
	coden = {MCMPAF},
	doi = {10.1090/S0025-5718-99-01093-5},
	fjournal = {Mathematics of Computation},
	issn = {0025-5718},
	journal = {Math. Comp.},
	mrclass = {65N30},
	mrnumber = {1648383 (99m:65211)},
	mrreviewer = {Zheng Hui Xie},
	number = {228},
	pages = {1379--1396},
	title = {Explicit error bounds in a conforming finite element method},
	url = {https://doi.org/10.1090/S0025-5718-99-01093-5},
	volume = {68},
	year = {1999}}

@incollection{Boss_hypercirc_a_post_98,
	author = {Bossavit, Alain},
	booktitle = {\'{E}quations aux d\'{e}riv\'{e}es partielles et applications},
	mrclass = {78M25 (65N30 78M10)},
	mrnumber = {1648224},
	pages = {221--237},
	publisher = {Gauthier-Villars, \'{E}d. Sci. M\'{e}d. Elsevier, Paris},
	title = {``{H}ypercercle'' et majorations d'erreur a posteriori par ``corrections locales''},
	year = {1998}}

@article{Cott_Diez_Huer_strict_lin_09,
	author = {Cottereau, R\'{e}gis and D\'{\i}ez, Pedro and Huerta, Antonio},
	doi = {10.1007/s00466-009-0388-1},
	fjournal = {Computational Mechanics},
	issn = {0178-7675},
	journal = {Comput. Mech.},
	mrclass = {74S05 (65N15)},
	mrnumber = {2520189},
	mrreviewer = {Shaochun Chen},
	number = {4},
	pages = {533--547},
	title = {Strict error bounds for linear solid mechanics problems using a subdomain-based flux-free method},
	url = {https://doi.org/10.1007/s00466-009-0388-1},
	volume = {44},
	year = {2009}}

@article{Par_San_Die_flux_a_post_09,
	author = {Par{\'e}s, N. and Santos, H. and D{\'{\i}}ez, P.},
	coden = {IJNMBH},
	doi = {10.1002/nme.2593},
	fjournal = {International Journal for Numerical Methods in Engineering},
	issn = {0029-5981},
	journal = {Internat. J. Numer. Methods Engrg.},
	mrclass = {65N30 (65N15)},
	mrnumber = {2560964 (2010j:65240)},
	mrreviewer = {Christian Wieners},
	number = {10},
	pages = {1203--1244},
	title = {Guaranteed energy error bounds for the {P}oisson equation using a flux-free approach: solving the local problems in subdomains},
	url = {https://doi.org/10.1002/nme.2593},
	volume = {79},
	year = {2009}}

@article{Mor_Noch_Sieb_stars_03,
	author = {Morin, Pedro and Nochetto, Ricardo H. and Siebert, Kunibert G.},
	doi = {10.1090/S0025-5718-02-01463-1},
	fjournal = {Mathematics of Computation},
	issn = {0025-5718},
	journal = {Math. Comp.},
	mrclass = {65N15 (65N50)},
	mrnumber = {1972728},
	mrreviewer = {Piotr P. Matus},
	number = {243},
	pages = {1067--1097},
	title = {Local problems on stars: a posteriori error estimators, convergence, and performance},
	url = {https://doi.org/10.1090/S0025-5718-02-01463-1},
	volume = {72},
	year = {2003}}

@article{Ched_Fuc_Priet_Voh_guar_rob_FE_RD_09,
	author = {Cheddadi, Ibrahim and Fu{\v{c}}{\'{\i}}k, Radek and Prieto, Mariana I. and Vohral{\'{\i}}k, Martin},
	doi = {10.1051/m2an/2009012},
	fjournal = {M2AN. Mathematical Modelling and Numerical Analysis},
	issn = {0764-583X},
	journal = {M2AN Math. Model. Numer. Anal.},
	mrclass = {65N15 (65N08 76R50)},
	mrnumber = {2559737 (2010i:65243)},
	number = {5},
	pages = {867--888},
	title = {Guaranteed and robust a posteriori error estimates for singularly perturbed reaction-diffusion problems},
	url = {https://doi.org/10.1051/m2an/2009012},
	volume = {43},
	year = {2009}}

\end{document}